\documentclass[12pt, reqno]{amsart}
\usepackage{amsmath}
\usepackage{amsxtra}
\usepackage{amscd}
\usepackage{amsthm}
\usepackage{amsfonts}
\usepackage{amssymb}
\usepackage {pstricks}
\usepackage{pstricks,pst-node}
\usepackage{tikz}

\usepackage{verbatim}

\def\1{\hbox{1\kern-.35em\hbox{1}}}

\def\abs#1{{\rm abs}(#1)}
\def\sign#1{{\rm sign}(#1)}

\newtheorem{claim}{\indent Claim}
\newtheorem{theorem}{Theorem}[section]
\newtheorem*{theorem*}{Theorem}
\newtheorem{lemma}[theorem]{Lemma}
\newtheorem{proposition}[theorem]{Proposition}
\newtheorem*{proposition*}{Proposition}

\newtheorem{definition}[theorem]{Definition}
\newtheorem{notation}[theorem]{Notation}
\newtheorem{remark}[theorem]{Remark}

\newtheorem{example}[theorem]{Example}

\newtheorem{fact}{Fact}
\newtheorem{Ofact}{\phantom{Fact}}

\numberwithin{equation}{section}

\newcommand{\bea}{\begin{eqnarray}}
\newcommand{\eea}{\end{eqnarray}}
\newcommand{\be}{\begin{eqnarray*}}
\newcommand{\ee}{\end{eqnarray*}}

\newcommand{\Z}{{\mathbb Z}}
\newcommand{\Q}{{\mathbb Q}}

\newcommand{\C}{{\mathbb C}}

\newcommand{\fb}{{\mathfrak b}}
\newcommand{\fg}{{\mathfrak g}}
\newcommand{\fh}{{\mathfrak h}}

\newcommand{\fl}{{\mathfrak l}}
\newcommand{\fn}{{\mathfrak n}}
\newcommand{\fp}{{\mathfrak p}}

\newcommand{\fu}{{\mathfrak u}}

\newcommand{\U}{{\rm U}}

\newcommand{\cO}{{\mathcal O}}
\newcommand{\cR}{{\mathcal R}}

\def\bigskip{\vskip6pt}
\def\medskip{\vskip6pt}

\def\a{\alpha}

\def\b{\beta}
\def\d{\delta}
\def\D{\Delta}
\def\g{\gamma}
\def\G{\Gamma}

\def\th{\varpi}
\def\th{\varphi}

\def\l{\lambda}
\def\L{\Lambda}

\def\si{\sigma}
\def\sc{\scriptstyle}
\def\ssc{\scriptscriptstyle}
\def\dis{\displaystyle}

\def\rar{\rightarrow}

\def\rrar{\rightarrow}
\def\llar{\leftarrow}

\def\atp{P_{\rm aty}}

\def\drar{\mbox{${\sc\,}\rrar\!\!\!\!\rar{\sc\,}$}}

\def\dlar{\mbox{${\sc\,}\lar\!\!\!\!\llar{\sc\,}$}}

\def\lar{\leftarrow}
\def\D{\Delta}

\def\bs{\backslash}

\def\ch{{\rm ch{\ssc\,}}}

\def\Z{\mathbb{Z}}

\def\Q{\mathbb{Q}}

\def\C{\mathbb{C}}

\def\es{\varepsilon}

\def\nex{{{}\hat{}}}
\def\pre{{{}\check{}{}}}

\def\wen#1{{^{(#1)}}}
\def\wen#1{{^{#1}}}

\def\equa#1#2{
\begin{equation}\label{#1}\mbox{$#2$}\end{equation}}
\def\equan#1#2{$$\mbox{$#2$}$$}

\numberwithin{equation}{section}

\def\cO{\mathcal{O}}

\begin{document}

\title[Generalised Jantzen filtration of Lie superalgebras]
{Generalised Jantzen filtration of Lie superalgebras II:  the exceptional cases}
\author[Yucai Su]{Yucai Su}
\address{Department of Mathematics, Tongji University, Shanghai  200092, China} \email{ycsu@tongji.edu.cn}
\author[R.B. Zhang]{R.B. Zhang}
\address{School of Mathematics and Statistics,
University of Sydney, Sydney, Australia}
\email{ruibin.zhang@sydney.edu.au}
\begin{abstract}
Let $\fg$ be an exceptional Lie superalgebra, and let $\fp$ be the
maximal parabolic subalgebra which
contains the distinguished Borel subalgebra
and has a purely even Levi subalgebra.
For any parabolic Verma module
in the parabolic category $\cO^\fp$, it is shown that
the Jantzen filtration is the unique Loewy filtration,
and the decomposition numbers of the layers of the filtration are determined by
the coefficients of inverse Kazhdan-Lusztig polynomials.
An explicit description of the submodule lattices of the parabolic Verma modules is given,
and formulae for characters and dimensions of the finite dimensional
simple modules are obtained.
\end{abstract}
\maketitle

\tableofcontents

\section{Introduction}\label{sect:intro}

In this paper we continue the investigation started in \cite{SZ3} on Jantzen type filtration
for classical Lie superalgebras \cite{K, Sch}.

The Jantzen filtration was first introduced by Jantzen \cite{J, J1} for Verma modules
over semi-simple complex Lie algebras in the BGG category $\cO$.
It soon became clear that its properties were deeply rooted in Kazhdan-Lusztig
theory. Much work was devoted to studying Jantzen filtration
(see, e.g.,  \cite{GJ, CIS, BC, I} and references therein) in the 80s,
culminating at the celebrated proof of
the Jantzen conjectures \cite{BB1, BB} by generalising geometric techniques
used in the proof \cite{BB1, BK} of the Kazhdan-Lusztig conjecture \cite{KL}.
The circle of ideas surrounding Jantzen filtration and its generalisations \cite{A1, A}
led to the development \cite{S, F, Str} (see \cite{H} for further references)
of deformation techniques capable of
reaching deeper properties of category $\cO$ which are otherwise difficult to see.

Let $\fg$ be a simple basic classical Lie superalgebra \cite{K, Sch} or $\mathfrak{gl}_{m|n}$ defined
over the field $\C$ of complex numbers.
We fix once for all the chain of subalgebras $\fh\subset\fb\subset\fp\subset\fg$,
where $\fh$ is a Cartan subalgebra, $\fb$ is the  distinguished Borel subalgebra
in the sense of \cite{K},  and $\fp=\fl\oplus\fu$ is the maximal parabolic subalgebra
with the nilradical $\fu$ and purely even Levi subalgebra $\fl$. Denote by $V(\lambda)$
the parabolic Verma module with highest weight $\lambda$ in the parabolic category
$\cO^\fp$ of $\Z_2$-graded $\fg$-modules. In \cite{SZ3},
we introduced a Jantzen type filtration
\[
V(\lambda)=V^0(\lambda)\supset V^1(\lambda) \supset V^2(\lambda)\supset \dots
\supset V^\ell(\lambda) \supset 0, \quad \text{where $V^\ell(\lambda)\ne 0$},
\]
for each parabolic Verma module $V(\lambda)$
by studying a deformation of the parabolic category $\cO^\fp$.   When $\fg$ is a type I
Lie superalgebra (consisting of $\mathfrak{gl}_{m|n}$, $A(k|l)$ and $\mathfrak{osp}_{2|2n}$),
we proved that
\begin{enumerate}
\item[(i)] the Jantzen filtration is the unique Loewy filtration of the parabolic Verma module $V(\lambda)$;
\item[(ii)] the decomposition numbers of the consecutive quotients of the Jantzen filtration
are described by the coefficients of the inverse Kazhdan-Lusztig polynomials; and
\item[(iii)] the length $\ell$ of the Jantzen filtration of $V(\lambda)$ is equal to the degree
of atypicality of
the highest weight $\lambda$.
\end{enumerate}
Recall that a filtration of a module is Loewy if its consecutive quotients are all semi-simple, and
it has the smallest length among all such filtrations.
The generalised Kazhdan-Lusztig polynomials are defined in terms of $\fu$-cohomology groups
$H^i(\fu, L(\lambda))$. The degree of atypicality of $\lambda$ is the number of
isotropic odd positive roots $\gamma_i$ such that $(\lambda+\rho,  \gamma_i)=0$ and
$(\gamma_i, \gamma_j)=0$.

In the present paper we show that if $\fg$ is any of  the exceptional Lie superalgebras
$D(2,1;a)$, $F_4$ and $G_3$, the Jantzen filtration of a parabolic Verma module
$V(\lambda)$ in $\cO^\fp$ satisfies properties (i) and (ii). In particular, parabolic Verma modules
are rigid. The precise statements of these results are given in
Theorems \ref{thm:rigid}, \ref{main-1} and
\ref{main-2}.

Property (iii) no longer holds for any exceptional Lie superalgebra.
For any $\lambda\in\fh^*$ which is integral dominant,
$V(\lambda)$ is not simple regardless of whether $\lambda$ is typical
(i.e., with the degree of atypicality being $0$) or not.
Thus property (iii) can not hold for such highest weights.
Another important fact is that category $\cO^\fp$ is not the category
of finite dimensional $\Z_2$-graded $\fg$-modules,
in sharp contrast to the type I case.

Let us list the other main results of this paper.

(1). We explicitly describe in Theorem \ref{verma-stru} the submodule lattice of
any parabolic Verma module $V(\lambda)$ in $\cO^\fp$ for all exceptional Lie superalgebras.
This is a significant result in its own right.  Here  it
enables us to prove the main results on Jantzen filtration.
The method used to prove Theorem \ref{verma-stru} is a generalisation of that in \cite{SZ4},
where we worked out the structure of the parabolic Verma modules
for $\mathfrak{osp}_{k|2}$. It involves finding the primitive vectors
in $V(\lambda)$.

(2). We compute the $\fu$-cohomology groups
with coefficients in any simple $\fg$-module with the help of Theorem \ref{verma-stru}.
The result is given in Theorem \ref{main-3}
(where we actually give the $\fu^-$-homology groups, but see Remark \ref{rem:HH}).
This result provides crucial information
needed for proving property (ii) of the Jantzen filtration of $V(\lambda)$.

(3). We determine the character and dimension
of any finite dimensional simple $\fg$-module by using Theorem \ref{verma-stru}, 
and give explicit formulae for them in Theorem \ref{main-theo1}.
Note that in the case of $D(2, 1; a)$, characters and dimensions of finite dimensional simple modules
were obtained in \cite{V}.

(4). We obtain in Theorem \ref{homology} the first and second $\fg$-cohomology groups
with coefficients in the finite dimensional simple
and Kac modules. Analogous results were obtained for the type I Lie superalgebras in \cite{SZ1} and for
$\mathfrak{osp}_{k|2}$ in \cite{SZ4}.

Finally we comment briefly on the method used in this paper. Recall that essential use was made of knowledge
on the generalised Kazhdan-Lusztig polynomials \cite{Se96, B, VZ} and super duality \cite{CZ, CWZ, CL, BS, CLW}
when establishing the above listed properties of the Jantzen filtration for the type I superalgebras in \cite{SZ3}.
For the exceptional Lie superalgebras,  no such results are available.
Thus in this paper, we will take a ``bottom-up" approach instead by understanding first the submodule lattice
of parabolic Verma modules then deducing the properties of the Jantzen filtration.

\section{Exceptional Lie superalgebras}
\label{Preliminaries}

In this section we present some preliminary material which will be needed
later.  Throughout the paper, we work over the field $\C$ of complex numbers.

Let $\fg$ be one of the exceptional Lie superalgebras, that is, $\fg$ is either
$D(2,1;a)$ ($a\ne0,-1$), $F_4$ or $G_3$.  Relative to the distinguished Borel subalgebra $\fb$,
the Cartan matrix of $\fg$ is given by
\[
A=\left(\begin{array}{ccc}0&1&a\\\!-1\!&2&0\\\!-1\!&0&2\end{array}\right)\begin{array}{c}0\\1\\2\end{array}, \ \ \left(\begin{array}{ccccc}2&\!-1\!&0&0\\\!-1\!&2&\!-2\!&0\\0&\!-1\!&2&1\\0&0&\!-1\!&0\end{array}\right)\begin{array}{c}1\\2\\3\\0\end{array}
\mbox{ \ or \ }
\left(\begin{array}{ccc}0&1&0\\\!-1\!&2&\!-3\!\\\!-1\!&0&2\end{array}\right)\begin{array}{c}0\\1\\2\end{array}
\]
(where the numbers beside the matrix are the row indices)
and Dynkin diagram by
$$
\stackrel{\a_0}{\otimes}\put(-3,2){$\line(2,-1){15}$}\put(-3,2){$\line(2,1){15}$}\put(0,-6){$\sc a$}\put(11,7){{\Large ${\circ}$}$\put(0,2){${\sc\a_1}$}$}\put(11,-12){{\Large ${\circ}$}$\put(0,3){${\sc\a_2}$}$}%
\ \ \ \ \ \ \ \ \ , \ \ \ \ \stackrel{\!\!\!\!\!\a_1\!\!\!\!\!}{\mbox{\Large$\!\!\!\!\circ\!\!\!\!$}}\put(-1,4){$\line(1,0){20}$}\put(18,0){$\stackrel{\!\!\!\!\!\a_2\!\!\!\!\!}{\mbox{\Large$\circ$}}$}\put(24,1){$\Longrightarrow
$}
\put(42,0){$\stackrel{\!\!\!\!\!\a_3\!\!\!\!\!}{\mbox{\Large$\circ$}}$}\put(49,4){$\line(1,0){20}$}\put(68,0){$\stackrel{\!\!\!\!\!\a_0\!\!\!\!\!}{\otimes}$}
\ \ \ \ \ \ \ \ \ \ \ \ \ \ \ \ \ \ \ \ \ \ \ \ \ \ \ \ \ \mbox{or \ \ } \ \ \ \ \stackrel{\!\!\!\!\!\a_0\!\!\!\!\!}{\!\!\!\otimes\!\!\!}\put(-1,4){$\line(1,0){20}$}\put(18,0){$\stackrel{\!\!\!\!\!\a_1\!\!\!\!\!}{\mbox{\Large$\circ$}}$}
%
%
\put(30,6){$\line(1,0){16}$}
\put(27,4){$\line(1,0){19}$}
\put(30,2){$\line(1,0){16}$}\put(25,0){\large$<$}
\put(45,0){$\stackrel{\!\!\!\!\!\a_2\!\!\!\!\!}{\mbox{\Large$\circ$}}$}\ \ \ \ \ \ \ \ \ \ \ \ \ \ \ .
$$

\noindent We realize the root system of $\fg$ in the vector space $E=\oplus_{i=0}^{I_1}\C\es_i$, ($I_1=2$ if $\fg$ is $D(2,1;a)$, and $3$ otherwise)  equipped with a non-degenerate symmetric bilinear form such that $\{\d:=\es_0,\,\es_i\,|\,i=1,...,I_1\}$ is an orthogonal basis satisfying
\begin{eqnarray}\label{bilnear}
\begin{aligned}
D(2,1;a):\  &(\d,\d)=-(1+a), \ (\es_1,\es_1)=1,\ (\es_2,\es_2)=a, 
\\
F_4:\  &(\d,\d)=-6,\ (\es_i,\es_j)=2\d_{ij}, 
 \ i,j=1,2,3,\\
G_3:\  &(\d,\d)=-2,\ (\es_i,\es_j)=\d_{ij}, 
 \ i,j=1,2,3.
\end{aligned}
\end{eqnarray}
Denote by $\Pi$ the set of the simple roots, by $\D_{\bar0}^+$  (resp. $\D_{\bar1}^+$) the set of even (resp. odd) positive roots.  For convenience, we write $\D_1^+$ as the disjoint union  $\D_1^+=\D_1^\ddag\cup\D_1^\pm$. Then we have
\begin{eqnarray}\nonumber
D(2,1;a): &&\Pi=\{\a_0=\d-\es_1-\es_2,\,\a_1=2\es_1,\,\a_2=2\es_2\},\\  \label{sys1}
	&&\D_{\bar0}^+=\{2\es_1,\ 2\es_2,\ 2\d\},\\ \nonumber
	&&\D_1^\ddag=\{\d+\es_1\pm\es_2\},\ \ \D_1^{\pm}=\{\d-\es_1\pm\es_2\}; \\ \nonumber
F_4: && \Pi=\{\a_1=\es_1-\es_2,\,\a_2=\es_2-\es_3,\,\a_3=\es_3, \,
\a_0=\frac12(\d{\sc\!}-{\sc\!}\es_1{\sc\!}-{\sc\!}\es_2{\sc\!}-{\sc\!}\es_3)\},\\  \label{sys2}
	&&\D_{\bar0}^+=\{\d,\,\es_i,\,\es_i\pm\es_j\,(i<j)\,|\,i,j=1,2,3\}, \\ \nonumber
	&&\D_1^\ddag = \Big\{\frac12(\d + \es_1 \pm \es_2 \pm \es_3)\Big\},\,
		\D_1^{\pm} = \Big\{\frac12(\d - \es_1 \pm \es_2 \pm \es_3)\Big\};\\ \nonumber
G_3:  &&\Pi=\{\a_0=\d-\es_1+\es_3,\,\a_1=\es_1-\es_2,\,\a_3=2\es_2-\es_1-\es_3\},\\   \label{sys3}
      &&\D_{\bar0}^+=\{2\d,\,\es_1 - \es_2,\,\es_1 - \es_3,\,\es_2 - \es_3,\, \\ \nonumber
	 &&\qquad\quad\,2\es_1 - \es_2 - \es_3,\,2\es_2 - \es_1 - \es_3,\,\es_1 + \es_2 - 2\es_3\},\\ \nonumber
       && \D_1^\ddag=\{\d+\es_1-\es_2,\,\d+\es_1-\es_3,\,\d+\es_2-\es_3\},\\ \nonumber
       && \D_1^\pm=\{\d,\,\d-\es_1+\es_2,\,\d-\es_1+\es_3,\,\d-\es_2+\es_3\}. \nonumber
\end{eqnarray}
Let $\D^+=\D_{\bar0}^+\cup\D_1^+$ and $\D=\D^+\cup(-\D^+)$.
We choose the  Chevalley generators $e_i=e_{\a_i}$, $f_i=f_{\a_i},\,h_i$, $i=0,...,I_2$, where $I_2=3$ if $\fg=F_4$ and $I_2=2$ otherwise. They satisfy the commutation relations
\begin{eqnarray}\label{Chevalley-1}\nonumber
&& [e_i,f_j]=\d_{ij}h_i,\\ \nonumber
&& [h_i,e_j]=\frac{2(\a_i,\a_j)}{(\a_i,\a_i)}e_j,\quad
[h_i,f_j]=-\frac{2(\a_i,\a_j)}{(\a_i,\a_i)}f_j, \ \  i\ne0,\\
&& [h_0,e_j]=\frac12(\a_0,\a_j)e_j,\quad
[h_0,f_j]=-\frac12(\a_0,\a_j)f_j.
\end{eqnarray}

We define the composite root vectors  $e_\a,f_\a$ with $\a\in\D^+\bs\Pi$ as follows: Let $i\in\{0,...,I_2\}$ be the smallest such that $\a=\b+\a_i$ for some $\b\in\D^+$, and set
$e_\a=[e_\b,e_{\a_i}],\,f_\a=[f_{\a_i},f_\b]$.  For instance, in the case of $D(2,1;a)$,
\begin{eqnarray}\label{ch-1}
\nonumber        &      &
e_{\d+\es_1-\es_2}:=[e_{1},e_{0}],\ f_{\d+\es_1-\es_2}:=[f_0,f_1],\\\nonumber
&      &
e_{\d-\es_1+\es_2}:=[e_2,e_{0}],\ f_{\d-\es_1+\es_2}:=[f_0,f_2], \\
 &      &
 e_{\d+\es_1+\es_2}:=[e_{\d-\es_1+\es_2},e_{1}]
		=[e_{\d+\es_1-\es_2},e_{2}]=[[e_{0},e_{1}],e_{2}],\  \\ \nonumber
&		&
f_{\d+\es_1+\es_2}:=[f_{1},f_{\d-\es_1+\es_2}],\\ \nonumber
 &      &
e_{2\d}:=[ e_{\d+\es_1+\es_2}, e_{0}]=-[ e_{\d+\es_1-\es_2}, e_{\d-\es_1+\es_2}],\ \\\nonumber
&		&
 f_{2\d}:=[f_{0},f_{\d+\es_1+\es_2}].
\end{eqnarray}
Then we have
\[
\fg=\fn^-\oplus\fh\oplus\fn^+,\ \ \mbox{ where \ } \fh=\bigoplus_{i=0}^{I_2}\C h_i,\ \fn^\pm=\bigoplus_{\a\in\D^+}\fg^{\pm\a},
\]
with $\fg^\a=\C e_\a$, $\fg^{-\a}=\C f_\a$.
There is an anti-involution $\tau$ of the universal enveloping algebra $U(\fg)$ such that
\begin{equation}\label{tau====}
\tau(\fg^\a)=\fg^{-\a},\quad  \tau(h)=h, \quad \a\in\D,\ h\in\fh.
\end{equation}
Here by $\tau$ being an anti-involution we mean that $\tau^2=1$ and $\tau(xy)=\tau(y)\tau(x)$ for all $x,y\in U(\fg)$.
Note that there is no sign factor on the right hand side of the second formula 
when both $x$ and $y$ are odd. One may introduce a sign factor to obtain a graded version of $\tau$, which 
however will not play any role in this paper. 


\begin{remark}\rm
It follows from the definition of $e_\a$ that $\tau(e_\a)$ is a nonzero scalar multiple of $f_\a$.
This scalar factor will not play any role in later computations.
\end{remark}

Observe that $\fg$ admits a $\Z$-gradation compatible with its $\Z_2$-gradation
\begin{eqnarray}\label{Z-grad}
\fg=\bigoplus_{k=-2}^2\fg_k, &\quad  \mbox{such that  }
{\rm deg\,}e_i=-{\rm deg\,} f_i=\d_{i,0},\ \ {\rm deg\,}h_i=0 , \\
\label{fg2===}
\fg_2=\C e_{\th}, &\quad \mbox{where $\th=\d$ if $\fg=F_4$, and  $\th=2\d$ otherwise.}
\end{eqnarray}
Here $\fg_0$  is a reductive Lie algebra with positive root system $\D_0^+=\D_{\bar0}^+\bs\{\th\}$:
\[
\fg_0
\cong\left\{\begin{array}{ll}
\C h_0\oplus sl_2\oplus sl_2&\mbox{if }\fg=D(2,1;a),\\[4pt]
\C h_0\oplus so(7)&\mbox{if }\fg=F_4,\\[4pt]
\C h_0\oplus G_2&\mbox{if }\fg=G_3.
\end{array}\right.
\]
We choose $h_\th\in\fh$ to be the unique element such that
\begin{equation}\label{h-theta}
[h_\th,e_\a]=\frac{2(\th,\a)}{(\th,\th)}e_\a\mbox{ \ for all }\a\in\D^+.
\end{equation}

For any $\l\in\fh^*$ and $h=\sum_{j=0}^{I_2}c_jh_j$ in $\fh$, we have
\begin{equation}\label{l(h)}
\l(h)=\frac{c_0}{2}(\a_0,\l)+\sum_{j=1}^{I_2}c_j\frac{2(\a_j,\l)}{(\a_j,\a_j)}.
\end{equation}
Write any $\l\in\fh^*$ in terms of the $\d\es$-basis as
\[
\l=\l_0\d+\sum_{i=1}^{I_1}\l_i\es_i=(\l_0\mid \l_1,\l_2,...,\l_{I_1}),
\]
and say that $\l_i$ is the {\it $i$-th coordinate} of $\l$ ($i=0,...,I_1$).  However,
we should bear in mind that in the case of $G_3$, the roots span a proper subspace of $E$,
and by \eqref{l(h)} the following elements have the same projections onto the subspace
\begin{equation}
\label{Other-weight-exp}
(\l_0\,|\,\l_1,\l_2,\l_3), \quad (\l_0\,|\,\l_1+x,\,\l_2+x,\,\l_3+x)\mbox{ \ for any }x\in\C\mbox{ \ if \ }\fg=G_3.
\end{equation}

Let $\rho_{\bar0}$ (resp. $\rho_1$) be half of the sum of the even (resp. odd) positive roots, and let
$\rho=\rho_{\bar0}-\rho_1$. Then we have the following table.\vskip5pt
\begin{eqnarray}\label{rho====}
\begin{array}{c|c|c|c}
\hline\\
\fg &  	\rho_{\bar0} & 	\rho_1 & 		\rho\\
\hline\\
D(2,1;a)& (1 \mid 1,1) &  	(2\mid 0,0) &  (-1 \mid 1,1) \\
\hline\\
F_4& 	\left(\frac12\,\Big|\,\frac52,\frac32,\frac12 \right)&  (2\,|\,0,0,0) &  \left(-\frac32\,\Big|\,\frac52,\frac32,\frac12\right) \\
\hline\\
G_3& 	(1\,|\,2,1,-3)&   \left(\frac72\,\Big|\,0,0,0\right) &  \left(-\frac52\,\Big|\,2,1,-3\right).\\
\hline
\end{array}
\end{eqnarray}
\vskip5pt
\noindent
Given any $\l\in\fh^*$, we let $\l^\rho=\l+\rho$ with coordinates denoted by $\l_i^\rho$, namely
\begin{equation}\label{l-rho}
\l^\rho=\l+\rho=(\l_0^\rho\,|\,\l_1^\rho,...,\l_{I_1}^\rho).
\end{equation}

Let $\si_i$ be the reflection on $E$ which acts by changing the sign of the $i$-th coordinate of
any element $\l$. In the cases of $F_4$ and $G_3$, the symmetry group ${\rm Sym}_3$ of rank 3
acts on  $E$ by permuting the coordinates $\l_1,\l_2,\l_3$ of any element $\l$.
Denote $\bar\si=\si_1\si_2\si_3$, which changes the signs of the coordinates $\l_1,\l_2,\l_3$ simultaneously.

Denote by $W_0$ and $W$ respectively the Weyl groups of $\fg_0$ and $\fg$. Then $W=W_0\times\langle\si_0\rangle$ and
\begin{eqnarray}\label{W0==}
&                    &
W_0=\left\{\begin{array}{lll}
\langle\si_1,\si_2\rangle
&\mbox{if } \fg=D(2,1;a),\\[4pt]
{\rm Sym}_3\ltimes\langle\si_1,\si_2,\si_3\rangle&\mbox{if }\fg=F_4,\\[4pt]
{\rm Sym}_3\times\langle\bar\si\rangle&\mbox{if }\fg=G_3.
\end{array}\right.
\end{eqnarray}
Here $\bar\si$ denotes its restriction on the span of the roots of $G_3$. Observe from \eqref{Other-weight-exp} that $\bar\si$ is in fact of length $2$ in the Weyl group.

As usual, we define the {\it dot action} of $W$ on $\fh^*$ by $w\cdot\l=w(\l+\rho)-\rho$  for any $w\in W$ and $\l\in\fh^*$. Let $\l^{\si_0}:=\si_0\cdot\l$. Then
\begin{equation}\label{l-si-0}
\l^{\si_0}:=(c-\l_0\,|\,\l_1,...,\l_{I_1}),
\end{equation}
with $c=2,3$ and $5$ for $\fg=D(2,1;a),F_4$ and $G_3$ respectively.

We have the usual partial order ``\,$\preccurlyeq$\,'' on $\fh^*$ given by  \begin{equation}\label{part-order}
\l\preccurlyeq \mu\ \ \Longleftrightarrow\ \ \mu-\l\in\Z_+\Pi=\Big\{\sum_{i=0}^{I_1}m_i\a_i\,\Big|\,m_i\in\Z_+\Big\}.
\end{equation}
Let $\G=\Z\D^+$ be the root lattice endowed with the lexicographical order, namely,
for $x=\sum_{i=0}^{I_1}x_i\a_i$ and $y=\sum_{i=0}^{I_1}y_i\a_i$,
\begin{equation}\label{order}
x<y\ \ \Longleftrightarrow\ \ \mbox{ for the first $i$ with $x_i\ne y_i$, we have $x_i<y_i$.}
\end{equation}

\section{Generalised Jantzen filtration}\label{filtration-sec}

Let $\fg$ be an exceptional Lie superalgebra, and let $\fp=\fl+\fu$ be the maximal parabolic subalgebra  with the purely even
Levi subalgebra $\fl=\fg_0$ and nilradical $\fu=\fg_{+1}+\fg_{+2}$.  Then
$\fp\supset\fb\supset\fh$, where $\fb$ is the distinguished Borel subalgebra.
We shall refer to $\fp$ as the {\it distinguished maximal parabolic subalgebra} for convenience.
Let $\fu^-=\fg_{-1}\oplus\fg_{-2}$, then $\fg=\fu^-\oplus\fp$. The notation will be in force throughout the paper.

\subsection{Parabolic category $\cO^\fp$}\label{sect:parabolic}

We consider the parabolic category $\cO^\fp$ of
$\Z_2$-graded $\fg$-modules, which are finitely generated over $U(\fg)$,
decompose into direct sums of weight spaces and are locally $U(\fp)$-finite.
Denote by $P_0$ the set of $\fg_0$-integral weights, and by $P^+_0$ the set of $\fg_0$-integral dominant weights.
Then
\[
\begin{aligned}
P_0&= \left\{\l\in\fh^*\, \left|\,\frac{2(\l, \alpha)}{(\alpha, \alpha)}\in\Z, \ \forall\, \alpha \in \D^+_0\right.\right\}, \\
P^+_0&= \left\{\l\in\fh^*\, \left|\,\frac{2(\l, \alpha)}{(\alpha, \alpha)}\in\Z_+, \ \forall\, \alpha \in \D^+_0\right.\right\},
\end{aligned}
\]
where  $\D^+_0$ is the set of the positive roots of $\fg_0$.
Given any $\l\in P^+_0$, the simple $\fp$-module $L^0(\l)$ with highest weight  $\l$ is necessarily finite dimensional.
We define the parabolic Verma module $V(\l)=U(\fg)\otimes_{U(\fp)}L^0(\l)$ over $U(\fg)$, which has a unique simple quotient $L(\l)$.  Then $V(\l)$ (and hence $L(\l)$) belongs to $\cO^p$ if and only if $\l\in P^+_0$. Furthermore,
$\{L(\l), \mathfrak{P}(L(\l))\,|\,  \l\in P^+_0\}$ is the set of non-isomorphic simple objects in $\cO^\fp$, where $\mathfrak{P}$ is the parity reversal functor.  Set
\begin{eqnarray}\label{eq:RR}
R_{\bar0}=\prod_{\a\in\D_{\bar0}^+}(e^{\a/2}-e^{-\a/2}), \quad
R_1=\prod_{\b\in\D_1^+}(e^{\b/2} +e^{-\b/2}).
\end{eqnarray}
One can immediately show that the character of the parabolic Verma module $V(\l)$ for any $\l\in P^+_0$ is given by
\begin{eqnarray}\label{char-Verma-for}
\ch V(\l) =\frac{R_1}{R_{\bar0}}\sum_{w\in
W_0}\sign{w}e^{w(\l+\rho)}.
\end{eqnarray}

Let $P=\left\{\l\in\fh^* \,\left|\,\frac{2(\l, \alpha)}{(\alpha, \alpha)}\in\Z, \ \forall\, \alpha \in \D^+_{\bar0}\right.\right\}$,
and call the elements {\it $\fg$-integral}. The elements of $P\cap P_0^+$ will be said to be {\it $\fg$-integral and $\fg_0$-dominant}.  Let $P^+$ be the subset of $P\cap P_0^+$ consisting of elements $\l$ which satisfy the following conditions.
Write $\l\in P^+$ in coordinates $\l=(\l_0\,|\,\l_1,...,\l_{I_1})$, then
\begin{eqnarray}\label{eq:int-domin}
\begin{aligned}
		&\quad	 \frac{2(\l, \alpha)}{(\alpha, \alpha)}\in\Z_+, \quad \forall\,\alpha \in \D^+_{\bar0}, \quad \text{and}\\
D(2,1;a):	& \quad      \l_0\ge2\mbox{ or }\l_0=1,\,\l_1+1=\pm a(\l_2+1)\mbox{ or }\l=0,\\
F_4:		&\quad       \l_0\ge2\mbox{ or }\l_0=\frac32,\,2(\l_1-\l_2-\l_3)+1=0\mbox{ or }\\
		&\quad \l_0=1,\,\l_1=\l_2,\l_3=0\mbox{ or }\l=0, \\
G_3:		&\quad       \l_0\ge3\mbox{ or }\l_0=2,\l_1=\l_2\mbox{ or }\l=0.
\end{aligned}
\end{eqnarray}
Call elements of $P^+$ {\it integral dominant} with respect to $\fg$. It is known \cite{Kac2, Kac3} that
\begin{equation}\label{fin-l-l}
\dim L(\l)<\infty \quad  \text{ if and only if } \  \l\in P^+.
\end{equation}
In this case,  we let  $\bar\l_0:=\frac{2(\th,\l)}{(\th,\th)}$, which is
$\l_0$ if $\fg=D(2,1;a)$ or $G_3$, and $2\l_0$ if $\fg=F_4$.
Then $U(\fg)f_{\th}^{\bar\l_0+1}v_\l$ is a submodule of $V(\l)$; the quotient module
\equa{Kac-module}{K(\l)=V(\l){\mbox{\Large{$/$}}}U(\fg)f_{\th}^{\bar\l_0+1}v_\l}
is usually referred to as the {\it Kac module} \cite{Kac2, Kac3},
which is the maximal finite dimensional quotient of $V(\l)$.

An element $\l\in P_0$ is {\it atypical} if there exists some isotropic odd root $\g\in\D_1^+$ such that $(\l^\rho,\g)=0$. In this case $\g$ is called an {\it atypical root} of $\l$.  If no atypical root exists, $\l$ is called {\it typical}.
Given any atypical weight $\l\in P_0$, denote by $\D_{\rm aty}(\l)\subset \D_1^+$ the set of its atypical roots.

The following result is due to Kac \cite{Kac2, Kac3}.
\begin{proposition}\label{typical-prop}
Assume that $\l\in P^+$ is  a typical weight, then
$L(\l)=K(\l)$, and \equa{typical-char} {\dis \ch L(\l)=\ch
K(\l)={\dis\frac{R_1}{R_{\bar0}}}\sum_{w\in W}
\sign{w}e^{w(\l+\rho)}. }
\end{proposition}

\smallskip

Recall that a module is said to be rigid if it has a unique Loewy filtration. We have the following result.

\begin{theorem}\label{thm:rigid}
The parabolic Verma modules in $\cO^\fp$ are rigid.
\end{theorem}
The theorem will be proven in Section \ref{sect:rigid}.

For any $\fg$-module $V$  in $\cO^\fp$, we use $V^\vee$ to denote the $\fg$-module 
which is the direct sum of the duals of weight spaces of $V$ with the $\tau$-twisted $U(\fg)$-action
defined for all $x\in U(\fg)$ and $f\in V^\vee$ by 
\[
(xf) (v) = f(\tau(x) v), \quad \forall v\in V. 
\]
Then $V^\vee$ is in $\cO^\fp$; in particular,  $V^\vee=V$ if $V$ is a finite dimensional simple module. 
This notion of dual modules will be used later.

\subsection{Generalised Jantzen filtration}

We recall from \cite{SZ3} some elements of
the generalised Jantzen filtration for parabolic Verma modules over Lie superalgebras.

Let $T:=\C[[t]]$ be the ring of formal power series in the
indeterminate $t$, and consider the category $\fg$-Mod-$T$ of $\Z_2$-graded $\U(\fg)$-$T$
bimodules such that the left action of $\C\subset\U(\fg)$ and right
action of $\C\subset T$ agree.
We take the  $\C$-algebra homomorphism $\phi:
\U(\fh)\longrightarrow T$ such that \begin{equation}\label{phiii}\phi(h)=t\d(h)\mbox{ \ for }h\in\fh.\end{equation}
This defines a $\fp$-action on $T$ (recall that $\fp=\fg_0\oplus\fg_1\oplus\fg_2$) by
\begin{eqnarray*}
h f = f\phi(h)\mbox{ \ and }
X f = 0\mbox{ \ for $f\in T,\,h\in\fh,\,X\in\fg^\a\subset\fp$.}
\end{eqnarray*}
Given any object $M$ in the category $\fg$-Mod-$T$,
we define the deformed weight space of
weight $\mu\in\fh^*$ by
\begin{equation}\label{deformed-w}
M_\mu = \{m\in M \mid h m = \mu(h) m + m\phi(h), \ \forall\,
h\in\fh\}.
\end{equation}
The deformed parabolic category $\cO^\fp(T)$ is the full
subcategory of $\fg$-Mod-$T$ such that each object
\begin{itemize}
\item is finitely generated over
$\U(\fg)\otimes_\C T$;

\item decomposes into the direct sum of deformed weight spaces; and

\item is locally $\U(\fp')$-finite, where $\fp'=[\fp, \fp]$.
\end{itemize}

The parabolic Verma modules are distinguished objects of
$\cO^\fp(T)$.
For $\lambda\in P_0^+$, let $L^0(\lambda)$ be the finite dimensional irreducible
$\fp$-module with highest weight $\lambda$. Introduce the $\fp$-module $L^0_T(\lambda)
=L^0(\lambda)\otimes_\C T$ with $\fp$ acting diagonally. This is
also a $\fp$-$T$-bimodule with $T$ acting on the right by
multiplication on the factor $T$. The parabolic
Verma module with highest weight $\lambda$ is
\begin{eqnarray}\label{deformed-Kac-module}
V_T(\lambda):=\U(\fg)\otimes_{\U(\fp)} L^0_T(\lambda)
=\U(\fg)\otimes_{\U(\fp)} \big(L^0(\lambda)\otimes_\C T\big).
\end{eqnarray}

Let $\tau$ be the anti-involution on the universal
enveloping algebra $\U(\fg)$ defined in \eqref{tau====}.
There exists a unique $T$-bilinear form
\begin{eqnarray}\label{form}
\langle\ , \ \rangle: V_T(\lambda)\times V_T(\lambda)\longrightarrow T
\end{eqnarray}
which is contravariant
\begin{eqnarray}\label{tau-prop}
\langle x m, m' \rangle= \langle m, \tau(x) m' \rangle, \quad
\forall \ m, m' \in V_T(\lambda), \ x\in\U(\fg),
\end{eqnarray}
non-degenerate in the sense that
\[
\langle m, V_T(\l) \rangle=\{0\} \quad \text{only when $m=0$},
\]
and normalised so that the highest weight vector $v_\l$ of $L^0(\l)$ satisfies
$\langle v_\l, v_\l\rangle=1$.

Let $V_T^i(\lambda) = \left\{m\in V_T(\lambda) \mid \langle m,
V_T(\lambda)\rangle\subset t^i\C[[t]] \right\}$ for each $i\in\Z_+$.
These are $\fg$-T submodules of $V_T(\lambda)$, which give rise to the following descending
filtration:
\begin{eqnarray}\label{filt-T}
V_T(\lambda)=V_T^0(\lambda)\supset V_T^1(\lambda)\supset
V_T^2(\lambda)\supset \cdots.
\end{eqnarray}

Regard $\C$
as a $T$-module with $f(t)\in\C[[t]]$ acting by multiplication by
$f(0)$. Let $\cR: \cO^\fp(T)\longrightarrow \cO^\fp$ be the specialisation
functor which sends an object $M$ in $\cO^\fp(T)$ to $M\otimes_T\C$ in
$\cO^\fp$, and a morphism $\psi: M\to N$ to
\[
\cR(\psi): M\otimes_T\C\to N\otimes_T\C, \quad \cR(\psi)(m\otimes_T
c) = \psi(m)\otimes_T c.
\]
Denote $V(\lambda)= $ $V_T(\lambda)\otimes_T\C$ and
$V^i(\lambda) = V^i_T(\lambda)\otimes_T\C$. Applying the
specialisation functor $\cR$ to \eqref{filt-T} we obtain the
following descending filtration
\begin{eqnarray}\label{filt}
V(\lambda)=V^0(\lambda)\supset V^1(\lambda)\supset
V^2(\lambda)\supset\cdots
\end{eqnarray}
for the parabolic Verma module $V(\l)$, which will be referred to as the {\em
Jantzen filtration} for $V(\lambda)$. The consecutive
quotients of the Jantzen filtration are
\begin{eqnarray}\label{filt-1}
V_i(\lambda)=V^i(\lambda)/V^{i+1}(\lambda), \quad i=0, 1, 2, \dots.
\end{eqnarray}
%
One of the main results on the Jantzen filtration is the following.
\begin{theorem}\label{main-1}
For any $\lambda\in P_0^+$, the Jantzen filtration of the parabolic Verma module $V(\lambda)$ is the unique Loewy
filtration.
\end{theorem}
We will prove the theorem in Section \ref{proof-main-1}.

\medskip

Let $H_i(\fu^-, L(\lambda))$ be the
$i$-th Lie superalgebra homology group of  $\fu^-:=\fg_{-1}\oplus\fg_{-2}$ with coefficients
in the restriction of $L(\lambda)$. Then $H_i(\fu^-, L(\lambda))$ admits a semi-simple $\fg_0$-action. For any  $\l, \mu\in P_0^+$,  define the following {\it generalised Kazhdan-Lusztig
polynomials} \cite{Se96}:
\begin{eqnarray}\label{poly}
p_{\lambda \mu}(q) = \sum_{i=0}^\infty (-q)^i [H_i(\fu^-, L(\lambda)):
L^0(\mu)],
\end{eqnarray}
where $[H_i(\fu^-, L(\lambda)): L^0(\mu)]$ is the multiplicity of
$L^0(\mu)$ in $H_i(\fu^-, L(\lambda))$.

\begin{remark}\label{rem:HH} \rm
The generalised Kazhdan-Lusztig polynomials may also be defined in terms of the $\fu$-cohomology groups $H^i(\fu, L(\lambda))$ instead \cite{SZ3}, as
\[
H^i(\fu, L(\lambda))=H_i(\fu^-, L(\lambda)^\vee)=H_i(\fu^-, L(\lambda)) \quad \text{ for all $\l\in P^+_0$},
\]
where $L(\lambda)^\vee$ is the $\tau$-twisted  dual of $L(\l)$ defined at the end of Section \ref{sect:parabolic}.
\end{remark}

Choose a linear order on $P_0^+$ compatible with the usual partial
order defined by the positive roots. Then the matrix
$P(q)=\Big(p_{\lambda \mu}(q)\Big)_{\lambda, \mu\in P^+_0}$ is upper
unitriangular. Let
$A(q)=\Big(a_{\lambda \mu}(q)\Big)_{\lambda, \mu\in P^+_0}$ be the
inverse matrix of $P(q)$, and refer to $a_{\lambda \mu}(q)$ as the
{\it inverse Kazhdan-Lusztig polynomials} of $\fg$.

For any $\lambda,
\mu\in P^+_0$, we also define
\begin{eqnarray}\label{J-poly}
J_{\lambda \mu}(q) =\mbox{$\dis \sum_{i=0}^\infty$}q^i [V_i(\lambda): L(\mu)],
\end{eqnarray}
where $[V_i(\lambda): L(\mu)]$ denotes the multiplicity of the
irreducible $\fg$-module $L(\mu)$ in $V_i(\lambda)$. They were referred to as
{\it Jantzen polynomials} in \cite{SZ3}.

Another main result on the Jantzen filtration is the following.
\begin{theorem}\label{main-2}
For any $\lambda, \mu\in P_0^+$,
the Jantzen polynomials $J_{\lambda \mu}(q)$ coincide with the
inverse Kazhdan-Lusztig polynomials $a_{\lambda \mu}(q)$.
\end{theorem}
We will prove  the theorem in Section \ref{sect:proof-JKL}.

The remainder of the paper is devoted to the proofs of Theorems \ref{thm:rigid},  \ref{main-1} and \ref{main-2}.

\section{Classification of atypical weights}\label{Aty-weight}

We introduce some combinatorics for integral atypical weights in this section.
In particular, we define two sets $\atp^1$ and $\atp^2$, in each of which
weights are related by ``up-moves" and ``down-moves" defined by Definition \ref{dei1}.
These two sets partition the set of  $\fg$-integral and $\fg_0$-dominant atypical weights.
The significance of the combinatorics will become clear in Section \ref{sect:structure-thm}, where it enables
us to describe the structure of parabolic Verma modules quite uniformly.

\subsection{The up and down moves}

We start by observing several simple facts.
An element $\l=(\l_0\,|\,\l_1,...,\l_{I_1})\in\fh^*$ is $\fg_0$-integral (resp., $\fg_0$-integral dominant) if and only if the following conditions are satisfied:
\begin{eqnarray}\label{fg-0-integral}
D(2,1;a):&       &
\l_1,\l_2\in\Z\mbox{ (resp., $\Z_+$)},   \nonumber\\
F_4:&       &2\l_i,\l_i-\l_j\in\Z\mbox{ (resp., $\Z_+$) for }1\le i<j,   \nonumber\\
G_3:&       &\l_1-\l_2,\frac13(2\l_2-\l_3-\l_1)\in\Z\mbox{ (resp., $\Z_+$)}.
\end{eqnarray}
Also,  $\l\in P_0^+$ is $\fg$-integral if  and only if $\l_0\in\Z$ in the cases of $D(2,1;a)$ and $G_3$, and $2\l_0\in\Z_+$ in the case of $F_4$.

\begin{proposition}\label{prop1}\label{prop:single}
\begin{enumerate}\item Assume $\fg=D(2,1;a)$ or $F_4$.
Let $\l\in P_0^+$ be an atypical weight with two atypical roots. Then $\l^\rho_0=0$; in particular, $\l$ is $\fg$-integral.
\item If $\fg=G_3$, then each atypical weight $\l\in P_0^+$ has only one atypical root.
\end{enumerate}
\end{proposition}
\begin{proof}
Assume  $\g_\pm\in\D_1^+$ are atypical roots of $\l$ with $\g_+\ne\g_-$.
First suppose $\fg=D(2,1;a)$.
Then $\l_1^\rho,\l_2^\rho\ge1$. Assume  $\g_\pm=\d+x_\pm\es_1+y_\pm\es_2$ for some $x_\pm,y_\pm$ $\in\{\pm1\}$. Then from $(\l^\rho,\g_+-\g_-)=0$, we obtain
$x'\l_1^\rho+ay'\l_2^\rho=0$, where $x'=x_+-x_-,\,y'=y_+-y_-$ $\in\{0,\pm2\}$ and $(x',y')\ne(0,0)$. This is impossible if $a\notin\Q$ (the field of rational numbers). If $a\in\Q$, we obtain $x'=\sign{a}y'=\pm2$ (where $\sign{a}$ is the sign of $a$), and thus $\g_\pm=\d\pm(\es_1-\sign{a}\es_2)$,  $\l_1^\rho=\abs{a}\l_2^\rho$ (where $\abs{a}$ is the absolute value of $a$) and $\l_0^\rho=0$.

Now assume $\fg=F_4$.
Let $\l\in P_0^+$ be an atypical weight. Then $\l_1^\rho>\l_2^\rho>\l_3^\rho>0$. As above, from $(\l^\rho,\g_+-\g_-)=0$, we can obtain $x_1\l_1^\rho+x_2\l_2^\rho+x_3\l_3^\rho=0$ for some $x_i\in\{0,\pm1\}$. From this, we deduce $x_2=x_3=-x_1=\pm1$, and  $\l^\rho=(0\,|\,\l_1^\rho,\l_2^\rho,\l_3^\rho)$ with  $\g_\pm=\frac12(\d\pm(\es_1-\es_2-\es_3))$.
This proves (1).

Assume $\fg=G_3$. 
We can suppose $\l_3=0$ and take $\rho=(-\frac52\,|\,5,4,0)$ by \eqref{Other-weight-exp}. Then we have $\l_1,\l_2,\l_1-\l_2,\frac13(2\l_2-\l_1)\in\Z_+$ by \eqref{fg-0-integral}.
Thus $\l_1^\rho=\l_1+5>\l_2^\rho=\l_2+4>\l_3^\rho$ $=0$.
Assume  $\g_\pm=\d+\es_{i_\pm}-\es_{j_\pm}$ ($1\le i_\pm\ne j_\pm\le 3$).
If $i_-=i_+$ or $j_-=j_+$ or $(i_-,j_-)=(j_+,i_+)$, then from $(\l^\rho,\g_+-\g_-)=0$, we
would obtain respectively $\l_{j_+}^\rho=\l_{j_-}^\rho$ ($j_+\ne j_-$) or $\l_{i_+}^\rho=\l_{i_-}^\rho$ ($i_+\ne i_-$) or $\l_{i_+}^\rho=\l_{j_+}^\rho$ ($i_+\ne j_+$), a contradiction. Thus $(i_-,j_-)=(j_+,k)$ or $(k,i_+)$ (where $i_+,j_+,k$ are pairwise distinct), but we would then obtain
$\l_1^\rho=2\l_2^\rho$, which would imply that $\l$ is singular. Thus the atypical root of $\l$ is unique, and (2) is proven.
\end{proof}

\begin{definition}\label{first-def}\rm
An element $\l\in P_0$ is called {\em regular} if there exists $w\in W_0$ such that $w\cdot\l\in P^+_0$, and in this case,  denote $\l^+=w\cdot\l$.  If $\l$ is not regular, it is called {\em singular}.
\end{definition}
\begin{definition}\label{dei1}\rm
Assume that $\l\in P_0$ has an atypical root $\g\in\D_1^+$.  Let $k$ (resp. $k'$) be the smallest positive integer rendering
$\l+k\g$ (resp., $\l-k'\g$) regular, and define
\[
\l\nex_\g =(\l+k\g)^+, \quad   \l\pre_\g=(\l-k'\g)^+ \quad \text{in $P^+_0$}.
\]
Call the procedure of obtaining $\l\nex_\g$ (resp. $\l\pre_\g$) from $\l$ an {\em up} $($resp. {\em down$)$ move along $\g$}.
If $\g$ is the only atypical root of $\l$, we simply write $\l\pre=\l\pre_\g,\,\l\nex=\l\nex_\g$.
\end{definition}

\begin{remark}\label{rem:central}\rm
For any $\nu\in\fh^*$, we denote by $\chi_\nu$ the central character determined by $\nu$.  If $\nu$ is typical,
then $\chi_\mu=\chi_\nu$ if and only if $\mu=w\cdot\nu$ for some $w\in W$. If $\nu$ is atypical, then
$\chi_\mu=\chi_\nu$ if and only if there exist atypical elements $\mu_i\in \fh^*$, $\g_i\in\D_{\rm aty}(\mu_i)$,
$t_i\in \C$, and $w_i\in W$ with $i=0, 1, \dots, k$ for some $k<\infty$ such that
\[
\mu_0=\nu, \quad \mu_{i+1}=w_i\cdot(\mu_i+t_i\g_i) \ \  \text{for all $i<k$}, \quad \mu_{k+1}=\mu.
\]
\end{remark}

\begin{remark}\rm\label{AnoRemark} It follows from Remark \ref{rem:central} that both $\l\nex_\g$ and $\l\pre_\g$ correspond to the same central character as $\l$.
\end{remark}

\begin{remark}\rm\label{rem:multi-moves}
Repeated applications of up (resp. down) moves to $\l$ produce the weights $(\l+k\g)^+$ (resp.  $(\l-k\g)^+$) for all $k>0$ such that $\l+k\g$ (resp.  $\l-k\g$) are regular.
\end{remark}

\subsection{Description of $\atp^1$ and $\atp^2$}
Now we define $\atp^1$ and $\atp^2$ for the exceptional Lie superalgebras case by case.

\subsubsection{The case $D(2,1;a)$}\label{sub-D21}
Take $\mu=-\rho=(1\,|\,-1,-1)$ (i.e., $\mu^\rho=0$), which is a singular atypical weight with an atypical root $\g=\d-\es_1-\es_2$ (in fact one can take $\g$ to be any root in $\D_1^+$). We set $\l^1=\mu{\,}\nex_{ \g}=(2\,|\,0,0),\,\l^{-1}=\mu{\,}\pre_{ \g}=0$. In general, for any $i\in\Z^*:=\Z\bs\{0\}$, we let
\begin{equation}\label{lambda-i}
\l\wen{i}= \left\{\begin{array}{l l}(i+1\,|\,i-1,i-1), &\mbox{if \ $i>0$},\\
(i+1\,|\,-i-1,-i-1), &\mbox{if $i<0$},
\end{array}
\right.
\end{equation}
and denote $\atp^1=\{\l^i\,|\,i\in\Z^*\}$.  Then one has
\begin{equation}\label{up-down-si}
(\l\wen{i})\nex=\l\wen{i+1}\ (i\ne-1),\ \ (\l\wen{-1})\nex
=\l\wen{1},\ \ (\l\wen{i})^{\si_0}=\l\wen{-i},\ \ \ (\l\nex)\pre=\l.
\end{equation}
[There is a slight difference between the notation here and that in \cite{SZ4}. Here $\l\wen{0}$ is undefined and the weight $\l\wen{i}$ for $i\le-1$ corresponds to  $\l^{(i+1)}$ in \cite{SZ4}.]

If $a\notin\Q$, we set $\atp^2=\emptyset$. Now assume $a=\frac p q$ for some coprime integers $p\ne0$ and $q>0$ (and $p\ne-q$).  For any fixed $x\in\Z_+\bs\{0\}$, we take $\l^0_\pm=\l^0$ to be the weight such that
\begin{equation}\label{l0-D21a)}
(\l^0)^\rho=\big(0\,\big|\,\abs{p}x,qx\big), \mbox{ \ i.e., }\l^0=\big(1\,|\,\abs{p}x-1,qx-1\big),
\end{equation}
which is a $\fg$-integral dominant atypical weight with two atypical roots $\g_\pm=\d\pm(\es_1-\sign{p}\es_2)$.
We use Definition \ref{dei1} to define
\begin{equation}\label{D21-lpm1}
\l_\pm^1=(\l^0)\nex_{\g_\pm}, \quad  \l_\pm^{-1}=(\l^0)\pre_{\g_\mp}.
\end{equation}
Then $ \l^1_+ = (\l^{-1}_+)^{\si_0} $,  $ \l^1_- = (\l^{-1}_-)^{\si_0}$,  and
\begin{eqnarray}\label{D21-l0}
\begin{aligned}
 \l^1_+    &=\left\{  \begin{array}{ll}
\big(2\,\big|\,\abs{p}x,qx - \sign{p} - 1\big)   &\mbox{if }x > 1\mbox{ or }q > 1\mbox{ or }p < 0,\\[4pt]
(3\,|\,p+1,0)   &\mbox{if }x=q=1,\,p>0,
\end{array}\right.\\
\l^1_-    &=\left\{  \begin{array}{ll}
\big(2\,\big|\,\abs{p}x - 2,qx + \sign{p} - 1\big)   &\!\!\!\mbox{if }x\! >\! 1\mbox{ or }p \!> \!1\mbox{ or }p\! <\! -1,\,q\! > \!1,\\[4pt]
(3\,|\,1,q+2\sign{p})   &\!\!\!\mbox{if }x \!=\! p\! =\! 1\mbox{ or }x\! =\! -p\! = \!1,\,q \!>\! 2,\\[4pt]
(4\,|\,1,0)   &\!\!\!\mbox{if }x=-p=1,\,q=2.
\end{array}\right.
\end{aligned}
\end{eqnarray}
 Now for $i\ge 2$, we define \begin{equation}\label{d21-li}\l_{\pm}^i=(\l^{i-1}_\pm)\nex,\ \ \l_{\pm}^{-i}=(\l^{1-i}_\pm)\pre.\end{equation}
Then $\l_\pm^i$'s are  $\fg$-integral and $\fg_0$-dominant atypical weights and  $(\l_\pm^i)^{\si_0}=\l_\pm^{-i}$.  We let  $\atp^2$ be the set
consisting of the weights $\l^0=\l^0_\pm$ and $\l^i_{\pm}$ ($i\in\Z^*$) for all $x\in\Z_+\bs\{0\}$.
\subsubsection{The case $F_4$}\label{sub-F4}

Let $ x\in\Z_+\bs\{0\}$ be fixed, and set  (cf.~\eqref{l0-D21a)}) \begin{equation}\label{mu===}\mu^\rho=(0\,|\,x,x,0), \mbox{ \ i.e., } \mu=\Big(\frac32\,\Big|\,x-\frac52,x-\frac32,\frac12\Big).\end{equation} Note that $\mu$ is a singular atypical weight
 with an atypical root 
 $\g=\frac12(\d+\es_1-\es_2-\es_3)$ (in fact one can take $\g$ to be any $\frac12(\d\pm(\es_1-\es_2)\pm\es_3)\in\D_1^+$ and obtain the same $\l^i$). Set $\l^1=\mu{\,}\nex_{ \g}$, $\l^{-1}=\mu{\,}\pre_{ \g}$, which are given by
\begin{equation}\label{l-1-F4-2}
\begin{aligned}
\l^1 &= \left\{  \begin{array}{ll}(2\,|\,x - 2,x - 2,0)   &\mbox{if }x \ge 2,\\[4pt]
(3\,|\,0,0,0)   &\mbox{if }x = 1,
\end{array}\right.\\
\l^{-1} &= \left\{  \begin{array}{ll}(1\,|\,x - 2,x - 2,0)   &\mbox{if }x \ge 2,\\[4pt]
0   &\mbox{if }x = 1.
\end{array}\right.
\end{aligned}
\end{equation}
For all $i\ge2$, let $\l^i=(\l^{i-1})\nex,\,\l^{-i}=(\l^{1-i})\pre$ (note that $\l^0$ is not defined). Then \eqref{up-down-si} holds.
Let $\atp^1$ be the set of all $\l^i$ with $i\in\Z^*$ and $x\in \Z_+\bs\{0\}$.

Let $1\le a_3<a_2\in\Z_+$ be fixed integers and set $a_1=a_2+a_3$ (thus $a_2\ge a_3+1\ge2$ and $a_1\ge3$). We take $\l^0_\pm=\l\wen{0}$ to be the weight such that (cf.~\eqref{l0-D21a)} and \eqref{mu===})\begin{equation}\label{l0-F4} (\l\wen{0})^\rho=(0\,|\,a_1,a_2,a_3), \mbox{ \ i.e., }\l^0=\Big(\frac32\,\Big|\,a_1-\frac52,a_2-\frac32,a_3-\frac12\Big),\end{equation} which is an atypical  $\fg$-integral dominant weight with two atypical roots
$\g_\pm:=\frac12(\d\pm(\es_1-\es_2-\es_3))$.
We define $\l^{\pm1}_\pm$ as in \eqref{D21-lpm1}. Then
\begin{equation}\label{l-1-F4-1}
\begin{aligned}
\l^1_+&=(2\,|\,a_1-2,a_2-2,a_2-1),\\
 \l_+^{-1}&=(1\,|\,a_1-2,a_2-2,a_2-1),\\
\l^1_- &= \left\{  \begin{array}{ll}(2\mid  a_1 - 3,a_2 - 1,a_3)   &\mbox{if }a_3 \ge 2,\\[4pt]
(\frac52\mid  a_2 - \frac32,a_2 - \frac32,\frac12)   &\mbox{if }a_3 = 1,
\end{array}\right.\\
\l^{-1}_- &= \left\{  \begin{array}{ll}(1\mid  a_1 - 3,a_2 - 1,a_3)   &\mbox{if }a_3 \ge 2,\\[4pt]
(\frac12\mid  a_2 - \frac32,a_2 - \frac32,\frac12)   &\mbox{if }a_3 = 1.
\end{array}\right.
\end{aligned}
\end{equation}
Now define $\l_\pm^{\pm i}$ for $i\ge2$ by \eqref{d21-li}, and let $\atp^2=\bigcup\left(\{\l^0=\l^0_\pm\}\cup\{\l^i_\pm\mid i\in\Z^*\}\right)$,
where the union is over all $a_i$ satisfying the given condition.

\subsubsection{The case $G_3$}\label{sub-G3}
As in the proof of Proposition \ref{prop1}, we always assume $\l_3=0$ for any weight $\l$.
Fix $ x\in\Z_+$, similar to \eqref{mu===}, we denote \begin{equation}\label{G3-mu}
\mu^\rho=\Big(\frac12\,\Big|\,3x+2,3x+1,0\Big), \mbox{ \ i.e., }\mu=(3\mid  3x-3,3x-3,0),\end{equation}
which is an atypical weight with the unique atypical root $\g=\d+\es_1-\es_2$.
If $x\ge1$, then $\mu$ is a $\fg$-integral dominant weight, and we set $\l^1=\mu$. If $x=0$, then $\mu,\mu+\g$ are singular (cf.~Definition \ref{first-def}), and we set $\l^1=\mu{}\nex=(\mu+2\g)^+=(5\mid  0,0,0)$.
Define $\l^{-1}:=(\l^1)^{\si_0}=(\l^1)\pre=(2\mid  3x-3,3x-3,0)$ if $x\ge1$ or $\l^{-1}=0$ otherwise.

For $i\ge2$, we define $\l^i=(\l^{i-1})\nex$ and $\l^{-i}=(\l^{1-i})\pre$ (there is no $\l^0$). Then \eqref{up-down-si} holds.
The following fact will be needed later:
\begin{equation}\label{l0-ige3}\l^i_0>\l^1_0\ge3\mbox{ \ for all \ }i\ge2.\end{equation}
 Denote  by $\atp^1$ the set of $\fg$-integral and $\fg_0$-dominant atypical weights obtained in this way for all $x\in\Z_+$, and  set $\atp^2=\emptyset$.

\subsection{A classification of  integral atypical weights}
Let $\fg$ be an exceptional Lie superalgebra. We have the following
classification of $\fg$-integral atypical weights.
\begin{proposition}\label{aty-weight-F4}
Let $\atp:=\atp^1\cup \atp^2$. Then every $\fg$-integral and $\fg_0$-dominant atypical weight belongs to $\atp$. Furthermore, a weight $\l^i\in \atp^1$ is $\fg$-integral dominant if and only if $i=-1$ or $i\ge1$, and
a weight $\l^i_\pm\in \atp^2$ is $\fg$-integral dominant if and only if $i\ge0$.
\end{proposition}
\begin{proof}
Suppose $\l$ is a $\fg$-integral and $\fg_0$-dominant atypical weight with an atypical root $\g'\in\D_1^+$. Then $\l_0\in\Z$ if $\fg=D(2,1;a)$ or $G_3$, and $\l_0\in\frac12\Z$ otherwise. Note that there exists a unique $k\in\Z$ such that the $0$-th coordinate
of $\nu^\rho:=\l^\rho+k\g$ is $\nu_0^\rho=0$ (if $\fg=D(2,1;a)$ or $F_4$) or $\nu_0^\rho=\frac12$ (if $\fg=G_3$).
Then either $\mu:=\nu=\nu^\rho-\rho$ is singular with atypical root $\g'$, or there exists a unique $\si\in W_0$ such that $\mu:=\si(\nu^\rho)-\rho$ is $\fg_0$-integral dominant with atypical root $\g=\si(\g')$. Thus $\mu$ is one of the weights used to generate elements in $\atp^1$ or $\atp^2$. It then follows from Remark \ref{rem:multi-moves} that $\l\in\atp$.
The second statement can be easily proven by inspecting \eqref{eq:int-domin}.
\end{proof}

Denote by $\atp^+= \atp\cap P^+$ the set of $\fg$-integral dominant atypical weights.  Let
\begin{equation}\label{Tail-aty}
\atp^t=\text{the set of all $\l^i, \, \l^i_\pm\in\atp$ with $ i<0$.}
\end{equation}
Every atypical weight $\l$ in $\atp^t$ satisfies $\l_0^\rho<0$, and in this case every $\g\in\D_{\rm aty}(\l)$ is called a {\it tail atypical root} of $\l$ (note that $\g\in\D_1^\pm$ if $\l\in\atp^1$ or $\l=\l^i_+\in\atp^2$). For $\l\in\atp\bs\atp^t$, an atypical root $\g\in\D_{\rm aty}(\l)$
 is called a {\it non-tail atypical root} of $\l$.

\section{Structure of parabolic Verma modules}

Let $\fg$ be an exceptional Lie superalgebra with the distinguished maximal parabolic subalgebra $\fp$ as defined
in Section \ref{sect:parabolic}.

\subsection{Primitive weight graphs}\label{graph}

We briefly recall  from \cite{SZ3} (see also \cite{SHK}) the notion of primitive weight graphs, which is needed in later sections.

Let $V$ be an object in $\cO^\fp$.
A nonzero $\fg_{0}$-highest weight vector $v\in V$ is called a {\em
primitive vector} if there exists a $\fg$-submodule $M$ of $V$ such
that $v\notin M$ but $\fu v\in M$. If we can take $M=0$, then
$v$ is a $\fg$-highest weight vector. The weight of a primitive vector is
called a {\em primitive weight}.
For a primitive weight $\mu$ of a $\fg$-module $V$, we shall use
$v_\mu$ to denote a nonzero primitive vector of weight $\mu$ which
generates an indecomposable submodule. Two primitive vectors
are regarded as the same if they
generate the same indecomposable submodule.

Denote by $P(V)$ the
multi-set of primitive weights of $V$, where the multiplicity of a primitive weight
$\mu$ is equal to the dimension of the
subspace spanned by all the primitive vectors with weight $\mu$.

For $\mu,\nu\in
P(V)$, if $\mu\ne\nu$ and $v_\nu\in \U(\fg)v_\mu$, we say that $\nu$
is {\em derived from} $\mu$ and write $\nu\dlar\mu$  or
$\mu\drar\nu$. If $\mu\drar\nu$ and there exists no $\l\in P(V)$
such that $\mu\drar\l\drar\nu$, then we say that $\nu$ is {\it
directly derived from} $\mu$ and write $\mu\rrar\nu$ or
$\nu\llar\mu$.
\begin{definition}\label{defi6.1}\cite{SZ3} \rm
We associate $P(V)$ with a directed graph, still denoted by $P(V)$,
in the following way: the vertices of the graph are elements of the
multi-set $P(V)$ (i.e., a primitive weight of multiplicity $m$ corresponds to $m$ distinct vertices).
Two vertices $\l$ and $\mu$ are connected by a
single directed edge pointing toward $\mu$ if and only if $\mu$ is
derived from $\l$. We shall call this graph the {\it primitive
weight graph of  $V$}.

The {\em skeleton} of the primitive weight
graph is the subgraph containing all the vertices and is such
that two vertices $\l$ and $\mu$ are connected by a single directed
edge pointing to $\mu$ if and only if $\mu$ is directly derived
from $\l$. In this case we say that the two weights are linked.
\end{definition}
Note that a primitive weight graph is uniquely determined by its skeleton.

A  {\it full subgraph} $S$ of $P(V)$ is a subset of $P(V)$ which
contains all the edges linking vertices of $S$.
If for any $\mu,\nu\in S$, we have that $\mu\drar\eta\drar\nu$ implies $\eta\in
S$, we call the full subgraph $S$ {\em closed}. It is clear that a
module is indecomposable if and only if its primitive weight graph
is {\it connected} (in the usual sense), and that a full subgraph of
$P(V)$ corresponds to a subquotient of $V$ if and only if it is
closed.
\begin{notation}\label{nota1}\rm
For a directed graph $\G$, we denote by $M(\G)$ any module with
primitive weight graph $\G$ if such a module exists. 
\end{notation}

Observe the following facts: If $\G$ is a
closed full subgraph of $P(V)$, then $M(\G)$ always exists, which is
a subquotient of $V$. Also, the primitive weight graph of $V^\vee$
is obtained from that of $V$ by reserving the directions of the edges.

\begin{remark}\label{rem:set} \rm
We may regard $P(V)$ as a set in such a way that any member $\l$ of multiplicity $m_\l>1$
will be considered as $m_\l$ distinct elements.
\end{remark}

\subsection{The $\fg_0$-highest weights in parabolic Verma modules}
We describe the set of $\fg_0$-highest weights in the atypical parabolic Verma module
$V(\l)$. This contains the set $P(V(\l))$ of primitive weights of $V(\l)$.

\begin{remark}\rm
We will show presently that every parabolic Verma module $V(\l)$ in $\cO^\fp$ is multiplicity free, namely, $\mu_\l=1$ for all $\l\in P(V)$.
\end{remark}

Given $\l\in P_0^+$, we define
\begin{eqnarray}\label{P-lambda+}
P_\l^+ &:= \{\mu \in  P_0^+\mid  \mu \preccurlyeq \l, \ \chi_\mu=\chi_\l\}.
\end{eqnarray}
If $M$ is a highest weight module for $\fg$ with highest weight $\l$, let
\begin{eqnarray}\label{P-0-V-l}
P_0(M)&:=\{\mu\in P^+_\l \mid  \text{$\mu$ is a
$\fg_0$-highest weight in $M$}\}.
\end{eqnarray}
Then $P(V(\l))\subset P_0(V(\l))\subset P_\l^+$ as every primitive weight $\mu$ of $V(\l)$ must correspond to the same central character as $\l$ does itself, i.e.,  $\chi_\mu=\chi_\l$.
Let
\begin{eqnarray} \label{a-l-m}
\begin{aligned}
a_{\l,\mu}=[V(\l): L(\mu)], \quad b_{\l,\mu}=[V(\l): L^{0}(\mu)],
\end{aligned}
\end{eqnarray}
where  $V(\l)$ is regarded as a $\fg_0$-module by restriction in the second formula.
Clearly $ a_{\l,\mu}\le b_{\l,\mu}$  for all $\mu\in P(V(\l))$.
%
%

In the following discussion, when $\l\in\atp^2$, we may assume $\l=\l_+^i$ for some $i$
as the case $\l=\l_-^i$ is exactly the same. Note that symbols $\l^0,\,\l^0_+,\l^0_-$  all denote the same weight $\l^0\in\atp^2$.
The following simple facts will be frequently used (cf.~\eqref{up-down-si} and \eqref{d21-li}):  for all $i$,
\begin{equation}\label{l-sigma==}
\begin{aligned}
\l\pre&=\left\{\begin{array}{ll}\l^{i-1}&\!\!\mbox{if $\l=\l^i\in\atp^1$ with $i\ne0,1$},\\[2pt]
\l^{-1}&\!\!\mbox{if $\l=\l^1\in\atp^1$},\\[2pt]
\l^{i-1}_\pm&\!\!\mbox{if $\l=\l^i_\pm\in\atp^2$ with $i\ne0$,}
\end{array}\right. \\
\l^{\si_0}&=\left\{\begin{array}{ll}\l^{-i}&\!\!\mbox{if }\l=\l^i\in\atp^1,\\[6pt]
\l^{-i}_\pm&\!\!\mbox{if }\l=\l^i_\pm\in\atp^2.\end{array}\right.
\end{aligned}
\end{equation}
\begin{lemma}\label{lemma-possible-primitive}
Let $\l\in\atp$. Then
\begin{enumerate}
\item
\begin{equation*}
P_\l^+=\left\{\begin{array}{lll}
\{\l^j\mid 0\ne j\le i\}&\mbox{if }\l=\l^i\in\atp^1,\\[4pt]
\{\l_+^j\mid j\le i\}\cup\{\l_-^{k}\mid  k<0\}&\mbox{if }\l=\l^i_+\in\atp^2\mbox{ with }i\ge0,\\[4pt]
\{\l^j_+\mid  j\le i\}&\mbox{if }\l=\l^i_+\in\atp^2\mbox{ with }i<0,\end{array}\right.
\end{equation*}
and similarly for $\l=\l^i_-\in\atp^2$.
\item\label{P-0-subset}
\begin{itemize}\item
if $\l=\l^0\in\atp^2$, then $P_0(V(\l))\subset\{\l^0,\,\l^{-1}_\pm\}$ and $b_{\l^0,\l^{-1}_\pm}\le1$;
\item  if $\l=\l^1\in\atp^1$, then $P_0(V(\l))\subset \{\l^1,\l^{-1},\l^{-2}\}$;
\item if $\l$ is not as above, let
$
\Omega_\l:=\{\l,\,\l\pre,\,
\l^{\si_0},\,(\l^{\si_0})\pre,\,(\l^{\si_0})\nex\}.
$
Then $($cf.~\eqref{l-sigma==}$)$
\begin{equation}\label{P0vl}  P_0(V(\l))\subset\Omega_\l, \ \mbox{ \ and \
$b_{\l,\l\pre}\le1$.} \end{equation}
In particular, if $\l\ne\l^0$ has a tail atypical root, then $P_0(V(\l))\subset\{\l,\l\pre\}$.
\end{itemize}
%
\end{enumerate}
\end{lemma}
\begin{proof}
(1) Part (1) can be verified directly by using Remark \ref{rem:central} and definitions of $\l^i,$ $\l_\pm^i$.

(2) It follows from the PBW
Theorem that $U(\fu^-)=U(\fg_{-1}\oplus\fg_{-2})$ has a basis
\begin{equation}\label{B====}
B=\Big\{f_{\Theta}=f_\th^{\theta_0}\mbox{$ \prod_{\a\in\D_1^+}$}f_\a^{\theta_\a}\,\Big|\,\Theta\in\Z_+\times\{0,1\}^{r}\Big\},
\end{equation}
where $r=\#\D_1^+$, $\Theta=\{\theta_0,\theta_\a\}_{\a\in\D_1^+}$, and the product in $f_\Theta$ is ordered so that
$f_\a$ is placed before $f_\b$ if $\a>\b$ for any $\a,\b\in\D_1^+$.
Define a total order on $B$ by
\equan{order-B}{f_{\Theta}>f_{\Theta'}\ \ \Longleftrightarrow\ \
|\Theta|>|\Theta'|\mbox{ \ or \ }|\Theta|=|\Theta'| \mbox{ but
}\Theta>\Theta',}
where $|\Theta|=\theta_0+\sum_{\a\in\D_1^+}\theta_\a$
is the {\it level} of $\Theta$, and $\Z_+\times\{0,1\}^{r}$ is ordered lexicographically.
Recall that a nonzero vector $v\in V(\l)$ can be uniquely written as
\equa{write-v}{v= b_1v_1+\cdots+b_tv_t,\ \ b_i\in B,\
b_1>b_2>\cdots, \ 0\ne v_i\in L^0(\l).} We call $b_1v_1$ the {\it
leading term} (cf.~\cite[\S5]{SHK}). A term $b_iv_i$ is called a
{\it prime term} if $v_i\in\C v_\l$. Note that a vector $v$ may have
zero or more than one prime terms. One can immediately prove the following facts
(cf.~\cite[Lemmas 5.1 and 5.2]{SHK} and \cite[Lemmas 3.5 and 3.6]{SZ4}).
\setcounter{Ofact}{1}\begin{fact}\rm\hangindent9ex\hangafter1
Let $v=gu$ for some $u\in V(\l)$ and $g\in U(\fu^-)$.\\
(1) If $u$
has no prime term then $v$ has no prime term.

\hspace*{6ex} \hangindent13ex\hangafter1(2) Let
$v'=gu',\,u'\in V(\l)$. If $u,u'$ have the same prime terms then
$v,v' $ have the same prime terms.
\vspace*{-7pt}
\end{fact}
\begin{Ofact}\rm\hangindent9ex\hangafter1
Let $v_\mu\in V(\l)$ be
a $\fg_0$-highest weight vector of weight $\mu$. Then
\equa{g0-highest}{\l-\mu=\theta_0\th+ \sum_{\a\in\D_1^+}\theta_\a\a,}
for some $\Theta=\{\theta_0,\theta_\a\}_{\a\in\D_1^+}\in\Z_+\times\{0,1\}^r.$ Furthermore, the
leading term $b_1v_1$ of $v_\mu$ must be a prime term.
\vspace*{-7pt}\end{Ofact}
\begin{Ofact}\rm\hangindent8.5ex\hangafter1 \label{gl0-highest-2}Suppose $v'_\mu=\sum^{t'}_{i=1}b'_iv'_i$ is
another $\fg_0$-highest weight vector with weight $\mu$. If all
prime terms of $v_\mu$ are the same as those of $v'_\mu$, then
$v_\mu=v'_\mu$.
\end{Ofact}

For any given $\mu\in P_0(V(\l))$, it follows from \eqref{g0-highest} that
\begin{eqnarray}\label{condi-mu-l}\nonumber
&         &\abs{\l_i-\mu_i}\le2,\,i=1,...,I_1\mbox{ \ if }\fg=D(2,1;a)\mbox{ or }F_4,\\
&         &\abs{\l_i-\mu_i}\le3,\,i=1,2\mbox{ \ \ \ \ \ \ if }\fg=G_3.
\end{eqnarray}
For $\fg=G_3$, we always assume that $\l_3=\mu_3=0$ (cf.~\eqref{Other-weight-exp}), and when an odd positive root like $\a=\d+\es_1-\es_3$ appears in the right-hand side of
\eqref{g0-highest}, we change it to $\d+2\es_1+\es_2$, as both represent the same weight by \eqref{Other-weight-exp}.
From \eqref{condi-mu-l}, we can verify directly (case by case) that $\mu\in\{\l^0,\l^{-1}_\pm\}$ if $\l=\l^0\in\atp^2$, and
$\mu\in\Omega_\l$ otherwise. Furthermore, in each of the following three cases:  (i) $\l^1\ne\l\in\atp^1$ and $\mu=\l\pre$, (ii)  $\l^0\ne\l\in\atp^2$ and $\mu=\l\pre$,  (iii) $\l=\l^0\in\atp^2$ and $\mu=\l^{-1}_\pm$ (in all these cases, $\l_0^\rho,\mu_0^\rho$ have the same sign or $\l_0^\rho$ is zero), the $\Theta$ in \eqref{g0-highest} is unique. By Fact 3, we obtain $b_{\l,\mu}\le1$. This proves Lemma \ref{lemma-possible-primitive}(2).
\end{proof}

\subsection{Structure theorem for parabolic Verma modules}\label{sect:structure-thm}

Now we prove the structure theorem of parabolic Verma modules. Recall that $P^+_0$ (resp. $P^+$) is the set of weights
which are integral dominant with respect to $\fg_0$ (resp. $\fg$).

\begin{theorem}\label{verma-stru}
Let $\fg$ be an exceptional Lie superalgebra, and $V(\l)$ be the parabolic Verma module with highest weight $\l\in P_0^+$.
\begin{enumerate}
\item \label{part1} Assume that $\l$ is typical.  Then $V(\l)$ is irreducible if $\l\not\in P^+$,
or has the primitive weight graph $\l\to\l^{\si_0}$ if $\l\in P^+$.
\item \label{part2} Assume that $\l$ is atypical.  If $\l\not\in P^+$, or  $\l\in P^+$ but has a tail atypical root, then $V(\l)$ has the primitive weight graph $\l\to\l\pre$.
\item \label{part3} Assume that $\l\in P^+$ is atypical with a non-tail atypical root $($i.e., $\l=\l^i\in$ $\atp^1$ with $i\ge1$ or $\l=\l^i_\pm\in\atp^2$ with $i\ge0\,)$, then the skeleton of the primitive weight graph for $V(\l)$ is one of the following directed graphs $($cf.~\eqref{l-sigma==} for symbols appearing in the last graph$)$:
\equa{Stru-Kac-mod}{\!\!\!
\begin{array}{c}        \l=\l\wen{\ssc 0}\in\atp^2\put(-55,-5){$\vector(-1,-2){10}$}\put(-55,-5){$\vector(1,-2){10}$}\put(-75,-40){$\l^{-1}_+$}\put(-45,-40){$\l^{-1}_-$}
\end{array}\!\! ,    \ \
\begin{array}{c}   \l=\l\wen{\ssc1}\put(0,0){$\in\atp^1$}\\[-4pt]\downarrow\\\l\wen{\ssc-2}\\ \end{array}\ \ \ \  \ \  ,\ \ \ \ \
\raisebox{15pt}{\mbox{$\begin{array}{c}            \l=\l\wen{\ssc2}\put(0,0){$\in\atp^1$}\\
\downarrow\put(0,0){$\searrow$}\\[-2pt]\ \ \,\l\wen{\ssc-1}\,\ \l\wen{\ssc1}\end{array}
\put(-17,-15){$\vector(-1,-2){8}$}\put(-32,-42){$\l\wen{\ssc-2}$}\put(-38,-15){$\vector(1,-2){8}$}
\put(-42,10){$\vector(-1,-3){20}$}\put(-68,-60){$\l\wen{\ssc-3}$}\put(-50,-47){$\vector(2,1){16}$}$}}\
\ \ \ , \ \ \ \ \ \ \ \ \ \ \raisebox{25pt}{\mbox{$\begin{array}{c}       \!\!\!\! \l\ \ \ \ \\
\phantom{\downarrow}\!\!\!\!\!\searrow\\[-2pt]\ \ \ \phantom{\l\wen{\ssc-1}}\!\! \l\pre\end{array}
\put(-17,-15){$\vector(-1,-2){8}$}
\put(-32,-42){$\l^{\si_0}$}
\put(-36,10){$\vector(-1,-3){18}$}\put(-68,-60){$(\l^{\si_0})\pre$}\put(-50,-47){$\vector(2,1){16}$}\put(-85,-76){\footnotesize $\sc
\l=\l^i\in\atp^1$ with $\sc i\ge3,$ or}\put(-85,-88){\footnotesize $\sc\l=\l^i_\pm\in\atp^2$ with $\sc i\ge1$}$}}
\ .}\end{enumerate}
\end{theorem}
\begin{proof} (1)
If $\l$ is typical and $\fg$-integral dominant, then using Remark \ref{rem:central}, one obtains $P_\l^+=\{\l, \l^{\si_0}\}$
from the definition \eqref{P-lambda+}. Since $V(\l)$ has at least two composition factors (see \eqref{Kac-module}),
$\l^{\si_0}$ appears in $P(V(\l))$. Now we determine its multiplicity.
By \eqref{Kac-module} and Proposition \ref{typical-prop}, the maximal submodule of $V(\l)$ is generated by $w_1:=f_\th^{\bar\l_0+1}v_\l$. So, a primitive vector with
weight $\l^{\si_0}$ has the following form
\begin{equation}\label{v'-llll}v'_{\l^{\si_0}}=uw_1=uf_\th^{\bar\l_0+1}v_\l\mbox{ \ for some }u\in U(\fg).
\end{equation}
Decompose $U(\fg)$ into
$U(\fg)=U(\fg^{-})U(\fg_1)U(\fg^{\ssc\ge0}_{_{\sc\bar0}})$,
where $\fg^{-}=\oplus_{\a\in\D^+}\fg^{-\a}$, $\fg_{_{\sc\bar0}}^{\ssc\ge0}=\fh\oplus\oplus_{\a\in\D_{\bar0}^+}\fg^\a$, and where we have adopted the convention that for any subspace $M$ of $\fg$, we use $U(M)$ to denote the subspace of  $U(\fg)$
spanned by PBW-monomials with respect to a fixed ordered basis of $M$. By \eqref{v'-llll} and the fact that $w_1$ is a $\fg_{\bar0}$-highest weight vector, we can choose $u\in$ $ U(\fg^{-})U(\fg_1)$. Since $w_1$ has weight $\mu:=\l-(\bar\l_0+1)\th$, and $\l^{\si_0}-\mu=\sum_{\a\in\D_1^+}\a=2\rho_1$, which is the maximal weight of $U(\fg_{1})=\wedge\fg_1$,
we see that $u$ has to be in $U(\fg_{1})$ with weight $2\rho_1$, i.e.,
\begin{equation}\label{u=====}
\mbox{$\dis u= \prod_{\a\in\D_1^+}e_\a $ \ \ up to a nonzero scalar factor,}
\end{equation}
where the order of the product is as specified in \eqref{B====}. [Actually the order does not matter.]
This proves $v'_{\l^{\si_0}}$ is unique, i.e., we have the graph $\l\to\l^{\si_0}$ in this case.

If $\l$ is typical but not $\fg$-integral dominant, then either $P_\l^+=\{\l\}$,  or $P_\l^+=\{\l,\l^{\si_0}\}$ if $\fg=G_3$ and with $\l_0$ being a half integer (otherwise $\l-\l^{\si_0}\notin\Z_+\Pi$, cf.~\eqref{part-order}). In the latter case, $\l$ is not $\fg$-integral, and one can verify that $V(\l)$ does not have a $\fg_{\bar0}$-highest weight vector with weight $\l^{\si_0}$. This proves (1).

(2) Next assume that $\l$ is atypical but not $\fg$-integral dominant,  or $\l$ is a $\fg$-integral dominant atypical weight with a tail atypical root. Note
from the proof of Lemma \ref{lemma-possible-primitive} that \eqref{P0vl}  holds even if $\l$ is not $\fg$-integral.
One can verify that either $V(\l)$ does not have a $\fg_{\bar0}$-highest weight vector with weight in $\Omega_\l\bs\{\l,\l\pre\}$, or
elements in $\Omega_\l\bs\{\l,\l\pre\}$  are not $\preccurlyeq\l$. Thus $P(V(\l))\subset\{\l,\l\pre\}$.
Let $w_2:=\prod_{\a\in\D_1^+}f_\a v_\l$ (with product being ordered as in \eqref{B====}),  which has weight $\l-2\rho_1$. One can prove that up to a nonzero scalar multiple (cf.~\cite[Equation (2.9)]{SZ3}),
\begin{equation}\label{WWWWW}
\mbox{$\dis \prod_{\a\in\D_1^+}e_\a w_2= \prod_{\a\in\D_1^+}e_\a \prod_{\a\in\D_1^+}f_\a v_\l  = \prod_{\a\in\D_1^+}(\l+\rho,\a)v_\l=0,
$}\end{equation}
where the last equality follows from the atypicality of $\l$. This implies that there is a
$\fg$-highest weight vector with weight $\prec\l$. Indeed, let $u\in U(\fg_{1})$ be the element with a maximal weight such that $w_3:=uw_2\ne0$, then $w_3$ is a $\fg$-highest weight vector. Since $P(V(\l))\subset\{\l,\l\pre\}$, the weight of $w_3$ must be $\l\pre$. Thus $1\le a_{\l,\l\pre}\le b_{\l,\l\pre}\le1$, and we have the graph $\l\to\l\pre$.
This proves  (2).

(3) Finally assume $\l$ is an atypical $\fg$-integral dominant weight with a non-tail atypical root.  Then $\l=\l^i\in\atp^1$ with $i\ge1$ or $\l=\l^i_\pm\in\atp^2$ with $i\ge0$.
For any $\fg_0$-integral dominant
weights $\l,\mu$ with $\mu\preccurlyeq\l$, we obtain from the character formulae for $\ch V(\l)$ and $\ch L^0(\mu)$ that
\begin{eqnarray} \label{b-l-mu=}
\dis b_{\l,\mu}= \sum_{(S, p, w)} \sign{w},
\end{eqnarray}
where the sum is over all triples $(S, p, w)\in \{S\subset\D_1^+\}\times \Z_+\times W_0$ such that
$\mu=w\cdot\nu$ with $\nu:=\l-\sum_{\a\in S}\a-p\th$  being  regular.
Note that these conditions in particular imply $\#S=|\l-\mu|-p$,
where $|\l-\mu|$ is the {\it relative level} which is  $2(\l_0-\mu_0)$ if $\fg=F_4$ and $\l_0-\mu_0$ otherwise. Although it is difficult to use \eqref{b-l-mu=} to compute $b_{\l,\mu}$ in general,  one can nevertheless obtain the following 
result
\begin{eqnarray}\label{b-l-mu===}
\nonumber&          &
b_{\l,\l\pre}=1 \quad \mbox{\ \ if $\l = \l^i \in \atp^1,\,i \ge 2$ \ or \ $\l = \l^i_\pm \in \atp^2,\,i \ge 1$},
\\\nonumber
&          &
b_{\l,\l^{-1}}=1 \quad \mbox{if $\l=\l^2\in\atp^1$},\\
&          &
 b_{\l,\l^{-1}}=0 \quad \mbox{if $\l=\l^1\in\atp^1$}.
\end{eqnarray}
[Some simplification takes place when $\mu=\l\pre$ or $\mu=\l^{-1}$, as the relative level $|\l-\mu|$ is comparatively small. For instance in the first case it is controlled by \eqref{condi-mu-l}.]

We now use this information and Lemma \ref{lemma-possible-primitive} to prove several claims, which will then imply the theorem.
\setcounter{claim}{0}
\begin{claim}\label{claim1}
For any weight $\mu$ appearing in \eqref{Stru-Kac-mod}, which is either $\fg$-integral dominant or appears in the first graph, we have $a_{\l,\mu}=1$.
\end{claim}
We already see $a_{\l,\mu}\le b_{\l,\mu}=1$. To prove $a_{\l,\mu}=1$,
first suppose $\l=\l^{i}$ or $\l^i_\pm$ for $i\gg0$. Then the only possible $\mu\prec\l$ is $\mu=\l\pre$, which is $\l^{i-1}$ or $\l^{i-1}_\pm$.
As in the proof of \eqref{WWWWW}, we see that there exists a $\fg$-highest weight vector with weight $\nu$ satisfying $\l-2\rho_1\preccurlyeq\nu\prec\l$ (this condition  implies that $\nu$ must be $\fg$-integral dominant when $\l_0\gg0$), and by Lemma \ref{lemma-possible-primitive}(2), $\l\pre\,$ is the only possible such $\mu$.
Thus $\l\pre\in P(V(\l^{(i)}))$ for all $i\gg0$.

Now assume conversely there exists some $i_0>0$ such that for $\l=\l^{i_0}$ or $\l^{i_0}_\pm$, $a_{\l,\mu}=0$ for some  said $\mu$ in the claim (i.e., either $\mu=\l\pre$, or else $\mu=\l^{-1}\in\atp^1$ with $i_0=2$). By the previous paragraph, we can choose $i_0$ to be maximal. Then we have the following facts:
\begin{itemize}
\item $V(\l)$ contains a $\fg_0$-highest weight $\mu$ by \eqref{b-l-mu===};
\item  any composition factor $L(\eta)$ of $V(\l)$  other than $L(\l)$
does not  contain  a $\fg_0$-highest weight $\mu$ simply because either $\eta\preccurlyeq\l\pre=\mu$ or else $i_0=2,\,\l=\l^2\in\atp^1,\,\eta=\l^{1},$ $\mu=\l^{-1}$ (in this latter case $L(\eta)=L(\l^1)$ cannot contain a $\fg_0$-highest weight $\mu=\l^{-1}$  as $V(\eta)=V(\l^1)$ does not contain one by \eqref{b-l-mu===}).
\end{itemize}
These facts imply that only the composition factor $L(\l)$ of $V(\l)$ contains  a $\fg_0$-highest weight $\mu$. Since $L(\l)$ is a composition factor of $V(\l\nex{\sc\,})$ by the maximal choice of $i_0$, we deduce that $V(\l\nex{\sc\,})$ contains a  $\fg_0$-highest weight $\mu$. However $\mu\notin\Omega_{\l^\nex}$, a contradiction with \eqref{P0vl}. This prove the claim except the case that $\mu$ appears in the first graph of \eqref{Stru-Kac-mod}.

To see $a_{\l^0,\l^{-1}_+}=1$, let us look at $V(\l^1_+)$, which contains the finite dimensional composition factor $L(\l^0)$ as we have just proved in the previous paragraph (note that $\l^0=(\l^1_+)\pre=(\l^1_-)\pre{\sc\,}$). Since the parabolic Verma module $V(\l^1_+)$ cannot contain a finite-dimensional submodule, we see that there must exist some weight $\nu\not\in P^+$ such that
\begin{equation}\label{unique-of-nu-in}
\l^0\drar\nu\mbox{ \ in the graph $P(V(\l^1_+))$}.
\end{equation}
Such a $\nu$ must be contained in $P(V(\l^1_+))$ and in $P(V(\l^0))$, this is because a module $M(\l^0\drar\nu)$ (cf.~Notation \ref{nota1}) with graph $\l^0\drar\nu$ must be a highest weight module with highest weight $\l^0$, and hence a quotient of $V(\l^0)$. By Lemma \ref{lemma-possible-primitive}(2), $\nu=\l^{-1}_+$ is the only possible weight. Thus
$a_{\l^0,\l^{-1}_+}=1$.
The uniqueness of $\nu$ in \eqref{unique-of-nu-in} in fact also implies \begin{equation}\label{l----0}\l^0\to\l^{-1}_+\mbox{ \ in the graph }P(V(\l^{0})),\end{equation}
i.e.,  this is the only possible connection between $\l^0$ and $\l^{-1}_+$ in $P(V(\l^0))$.
Similarly, considering $V(\l^1_-)$ instead of $V(\l^1_+)$ implies that we have $\l^0\to\l^{-1}_-$ in $P(V(\l^{0}))$.
This not only completes the proof of Claim \ref{claim1} but also proves the first graph of \eqref{Stru-Kac-mod}.

\begin{claim}\label{claim1+}
We have the second graph of \eqref{Stru-Kac-mod}.\end{claim}

In this case, $\l=\l^1\in\atp^1$, and we have $\Omega_\l=\{\l^1,\l^{-1},\l^{-2}\}$. However, $\l^{-1}\notin P(V(\l^1))$
by \eqref{b-l-mu===}, so $P(V(\l^1))\subset\{\l^1,\l^{-2}\}$ (thus equality must hold
since $V(\l^1)\ne L(\l^1)$). Let $v'_{\l^{-2}}\in V(\l^1)$ be any primitive vector with weight $\l^{-2}$. As in \eqref{v'-llll}, we have
 \begin{equation}\label{l-2v'}\mbox{ $v'_{\l^{-2}}=u_1w_1$ for some $u_1\in U(\fg^{-})U(\fg_1)$, where $w_1=f_\th^{\bar\l_0+1}v_\l$.}\end{equation}
Note from \eqref{v'-llll} that $w_1$ has weight $\l^{-1}-2\rho_1$, and so\begin{equation}\label{wei-u1}\mbox{ $u_1$ has weight $2\rho_1-(\l^{-1}-\l^{-2})$.}\end{equation} Also note
that $\prod_{\a\in\D_1^+}e_\a w_1=0$ as otherwise it would be a $\fg$-highest weight vector with weight $\l^{-1}$ (cf.~\eqref{u=====}).
Let $x\in U(\fg_1)$ be an element with maximal weight $\xi$ such that $xw_1\ne0$ (so $\xi_1:=2\rho_1-\xi\succ0$). By definition of $x$, the vector $xw_1$ is a $\fg$-highest weight vector with weight
$\eta:=\l^{-1}-\xi_1$ (so $\eta\in P(V(\l^1))$), thus the only possible $\eta$ is $\eta=\l^{-2}$. Hence $\xi=2\rho_1-(\l^{-1}-\l^{-2})$, this together with \eqref{wei-u1} proves that $u_1$ and $x$ has the same weight.
Then the unique choice of $x$ also shows that $u_1\in U(\fg_1)$. 
From Definition \ref{dei1}, we see that $\l^{-1}-\l^{-2}$ can be uniquely written as
a sum of distinct roots in $\D_1^+$, accordingly, the element in $U(\fg)$ with weight $\xi$ is unique up to a nonzero scalar factor. This proves that $u_1=x$ is unique up to a nonzero scalar factor (so $a_{\l^1,\l^{-2}}=1$), and we have the second graph of \eqref{Stru-Kac-mod}.\vskip5pt

From now on we shall consider the last two graphs of  \eqref{Stru-Kac-mod}, i.e., $\l=\l^i\in\atp^1$ with $i\ge2$ or $\l=\l^i_\pm\in\atp^2$ with $i\ge1$.
\begin{claim}\label{claim2--}
For any weight $\mu\notin P^+$ which does not appear in any of the last two graphs of \eqref{Stru-Kac-mod},
$a_{\l,\mu}=0$.\end{claim}

Note that the only possible weight in $\Omega_\l$ which does not appear in the graphs is the weight $\mu=(\l^{\si_0})\nex=\l^{1-i}$ with $\l=\l^i\in\atp^1$ and $i\ge3$ or $\mu=(\l^{\si_0})\nex=\l^{1-i}_\pm$ with $\l=\l^i_\pm\in\atp^2$ and $i\ge2$. In either case, $(\l^{\si_0})\nex\not\in P^+$.
Assume $v'_{(\l^{\si_0})\nex}$ is a primitive  vector with weight $(\l^{\si_0})\nex$. Then $v'_{(\l^{\si_0})\nex}\in U(\fg)f_\th^{\bar\l_0+1}v_\l$, but
the maximal possible weight of $U(\fg)f_\th^{\bar\l_0+1}v_\l$ is  $\l^{\si_0}$ (as we have seen in \eqref{v'-llll}), which is $\prec(\l^{\si_0})\nex$, a contradiction. Thus $a_{\l,\mu}=0$.

\begin{claim}\label{claim2-}
If a weight $\mu$ appearing in any one of the last two graphs of \eqref{Stru-Kac-mod} satisfies the conditions $\mu\ne\l$ and $\mu\in P^+$, then we have the subgraph $\l\to\mu$ in the graph $P(V(\l))$.
\end{claim}

We have $a_{\l,\mu}=1$ by Claim \ref{claim1}. If there exists some weight $\nu$ with $\l\drar\nu\drar\mu$  in the graph  $P(V(\l))$, then $\nu\not\in P^+$ by
Lemma \ref{lemma-possible-primitive}(2)
(otherwise we would have $\l=\l^2\in\atp^1$ and either $\nu=\l^1,\,\mu=\l^{-1}$ or else $\nu=\l^{-1},\,\mu=\l^{1}$, but we already know there does not exist a graph $\l^1\drar\l^{-1}$ nor $\l^{-1}\drar\l^1$ as a module $M(\l^1\drar\l^{-1})$ with the first graph would be  a quotient of $V(\l^1)$ and the second would be some kind of dual of the first, namely, $M(\l^{-1}\drar\l^1)
=M(\l^1\drar\l^{-1})^\vee$).
Therefore the primitive vector $v'_\nu$ with weight $\nu$ must be
 in $U(\fg)f_\th^{\bar\l_0+1}v_\l$ and thus, so is the primitive  vector $v'_\mu$ with weight $\mu$. This means that $L(\mu)$ is not a composition factor of the Kac-module $K(\l)$, a contradiction with the maximality of $K(\l)$.\vskip5pt

From now on we assume $\l=\l^i\in\atp^1$ with $i\ge2$ as the proof for the case $\l=\l^i_\pm\in\atp^2$ with $i\ge1$ is analogous.

\begin{claim}\label{claim2}We have the subgraph $\l^{i-1}\to\l^{-i}$ in the graph  $P(V(\l^i))$.
\end{claim}
Recall from \eqref{l-sigma==}  that $\l^{-i}=\l^{\si_0}$. First suppose $v'_{\l^{-i}}$ is a primitive vector with weight $\l^{-i}$, which must be in $U(\fg)f_\th^{\bar\l_0+1}v_\l$.
Thus as in the proof of \eqref{v'-llll} and \eqref{u=====}, such $v'_{\l^{-i}}$ is unique, i.e., $a_{\l^i,\l^{-i}}\le1$. We already know $\l^{i-1}=\l\pre\in P(V(\l))$,
 thus there must exist some weight $\nu\not\in P^+$ such that
\begin{equation}\label{uniq---l}
\mbox{$\l^{i-1}\drar\nu$ \ in $P(V(\l^i))$,}
\end{equation}
 this is because $L(\l^{i-1})$ is finite dimensional and $V(\l^i)$ does not contain a finite dimensional submodule. As in the proof of \eqref{l----0},
such a weight $\nu$ is in $P(V(\l^i))\cap P(V(\l^{i-1}))$ 
 and thus must be $\l^{-i}$ (thus $a_{\l^i,\l^{-i}}\ge1$), and furthermore, the uniqueness of $\nu$ in \eqref{uniq---l} also implies
 \begin{equation}\label{l----0+++}\l^{i-1}\to\l^{-i}\mbox{ \ in the graph  $P(V(\l^i))$,  and no $\eta$ with $\l^{-i}\drar\eta$}.\end{equation}

\begin{claim}\label{claim3}We have  the subgraph $\l^{i}\to\l^{-i-1}\to\l^{-i}$ in the graph  $P(V(\l^i))$.\end{claim}

Loot at the graph  $P(V(\l^{i+1}))$, by Claim \ref{claim2} or \eqref{l----0+++}, we have a subgraph $\l^i\to\l^{-i-1}$. Since a module $M(\l^i\to\l^{-i-1})$ with  graph $\l^i\to\l^{-i-1}$ must be a quotient of $V(\l^i)$, we obtain that
 $\l^{-i-1}\in P(V(\l^i))$ and $\l^i\to\l^{-1-i}$ is a subgraph of $P(V(\l^i))$. Now assume $v'_{\l^{-1-i}}$ is a primitive vector with weight ${\l^{-1-i}}$. We want to prove
 $v'_{\l^{-1-i}}$ is unique.

  Let $M_1$ be the submodule of $V(\l^i)$ generated by the primitive vector with weight $\l^{-i}$, and set $M=V(\l^i)/M_1$.
Note from \eqref{l----0+++} that $M_1$ does not have a primitive weight $\l^{-i-1}$ (in fact $M_1$ is the simple module $L_{\l^{-i}}$ by \eqref{l----0+++}), thus  $v'_{\l^{-1-i}}$ uniquely corresponds to a primitive vector (also denoted by the same symbol)
in $P(M)$. Now as in the proof of Claim \ref{claim1+}, such a primitive vector in $P(M)$ is unique, and it must have the form
 $v'_{\l^{-1-i}}=u_1w_1$ for some $u_1\in U(\fg_1)$ (cf.~\eqref{l-2v'}), i.e.,  $u_1=\prod_{\a\in S}e_\a$ for some subset $S$ of $\D_1^+$. Now return to the parabolic Verma module $V(\l^i)$, we obtain that a primitive vector with weight $\l^{-1-i}$ is unique, which is $v'_{\l^{-1-i}}=\prod_{\a\in S}e_\a w_1$.
 Since the primitive vector with weight $\l^{-i}$ is  $v'_{\l^{-i}}=\prod_{\a\in\D_1^+}e_\a w_1$ (as in the proof of Claim \ref{claim2}),
 which can be then written as $v'_{\l^{-i}}=u_2v'_{\l^{-i-1}}$ with $u_2=\prod_{\a\in\D_1^+\bs S}e_\a$, i.e.,
$\l^{-i-1}\drar\l^{-i}$ in $P(V(\l^i))$. As there is no possible primitive weight sitting in between
$\l^{-i-1}$ and $\l^{-i}$, we have $\l^{-i-1}\to\l^{-i}$.

We have completed the proof of Theorem \ref{verma-stru}.
\end{proof}

\section{Proofs of main theorems on Jantzen filtration}\label{filtration-sec1}

\subsection{Rigidity of parabolic Verma modules}\label{sect:rigid}

Note that one can read the radical filtration of a module in $\cO^\fp$ off its primitive weight graph.
To see this, let $P(V)$ be the multi-set of primitive weights of a module $V\in \cO^\fp$, which will be regarded
as a set in the way explained in Remark \ref{rem:set}.
We decompose the set into a disjoint union of subsets
$P(V)=\cup_{i\ge 0} P(V)_i$ in the following way.
The subset $P(V)_0$ consists of the primitive weights which are not
derived from any weights. A primitive weight $\mu$ belongs to $P(V)_i$ if
in the skeleton of $P(V)$,
the longest of the oriented paths from weights in $P(V)_0$ to $\mu$ has $i$ arrows.
Let $V^{(i)}$ be the submodule of $V$ generated by all primitive
vectors in $\cup_{j\ge i}P(V)_j$.  Then we obtain the following
filtration for $V$,
\begin{eqnarray}\label{FILT---}
V=V^{(0)}\supset V^{(1)}\supset\cdots \supset
V^{(\ell)}\supset V^{(\ell+1)}=0,
\end{eqnarray}
where the {\it length} $\ell$ of the filtration  is the smallest non-negative integer such that $P(V)_{\ell+1}=\emptyset$.
We will say that elements of $P(V)_i$ are at level $i$.
\begin{lemma}\label{lem:graph} The filtration
\eqref{FILT---} is the radical filtration of $V$.  Furthermore, the consecutive quotients
of the filtration are given by
\[
V_i:= V^i/V^{i+1} = \bigoplus_{\l\in P(V)_i} L(\l), \quad i=0, 1, \dots, l.
\]
\end{lemma}
\begin{proof}
This  is obvious from the definition of a primitive weight graph.
\end{proof}

\begin{example}\label{exmple-graph}\rm
Consider as an example a module $V$ with the skeleton of its primitive weight graph given by Figure \ref{fig1}.

\begin{figure}[h]
\begin{center}
\begin{picture}(160, 80)(0,0)
\put(80, 80){$\l$}
\put(80, 80){\vector(-1, -2){40}}
\put(32, -10){$\mu'$}

\put(42, 2){\vector(3, 1){50}}
\put(88, 7){$\nu$}

\put(80, 80){\vector(1, -1){30}}
\put(113, 40){$\mu''$}

\put(80, 80){\vector(-1, -4){10}}
\put(60, 30){$\mu$}

\put(70, 40){{\vector(1, -1){20}}}
\put(110, 50){{\vector(-2, -3){20}}}
\end{picture}
\end{center}
\caption{Example}
\label{fig1}
\end{figure}
 The radical filtration of $V$ has length $2$, with the non-empty $P(V)_i$ given by
\[
P(V)_0=\l, \quad P(V)_1=\{\mu, \mu', \mu'' \}, \quad P(V)_2=\nu,
\]
and the consecutive quotients of the radical filtration given by
\[
V_0=L(\l),  \quad V_1=L(\mu)\oplus L(\mu')\oplus L(\mu''), \quad V_2=L(\nu).
\]
The radical filtration is the unique Loewy filtration in this case, as one can immediately see by looking at the graph.
\end{example}

Now we can easily prove Theorem \ref{thm:rigid}.
\begin{proof}[Proof of Theorem \ref{thm:rigid}]
We can construct the radical filtration of $V(\l)$ by applying Lemma \ref{lem:graph} to its primitive weight graph,
which is one of the graphs obtained in Theorem \ref{verma-stru}.  By  inspecting the graph, we immediately see that
the radical filtration of $V(\l)$ is the unique Loewy filtration.
Note that the third graph in \eqref{Stru-Kac-mod} is a special case of Figure \ref{fig1},
which was already considered in detail in Example \ref{exmple-graph}.
\end{proof}

\subsection{Proof of  Theorem \ref{main-1}}\label{proof-main-1}
The claim of Theorem \ref{main-1} is trivially true if $V(\l)$ is simple, thus we assume that $V(\l)$ is reducible. We only need to show that the Jantzen filtration is Loewy since $V(\l)$ is rigid by Theorem \ref{thm:rigid}.
From Theorem \ref{verma-stru} we can see that $V(\l)$ is  multiplicity free, namely, all composition factors have multiplicity one.
It then follows from the proof of \cite[Theorem 3.6]{SZ3} that the Jantzen filtration has semi-simple consecutive quotients.

It remains to prove that the length $\ell$ (cf.~\eqref{FILT---}) of the Jantzen filtration is minimal, i.e., $\ell=1$ in all cases except
that $\ell=2$ if $\l$ is in the cases of the last two graphs of \eqref{Stru-Kac-mod}.

First we suppose that $\l$ is a $\fg$-integral and $\fg_0$-dominant atypical weight such that $\l=\l\wen{i},\l^i_\pm$ with $i\ge2$ or $\l=\l^1_\pm\in\atp^2$ (i.e., $\l$ is in the last two cases of \eqref{Stru-Kac-mod}). We assume $\l=\l^i$ as the proof for the case $\l=\l^i_\pm$ is similar.
We only need to prove that the socle $L(\l^{-i})$  of $V(\l^i)$ is contained in $V^2(\l)$.

Let $v'_{\l^{-i}}$ be a primitive vector with weight $\l^{-i}$. Then it is a nonzero highest weight vector of
$L(\l^{-i})\subset V(\l^i)$. Up to a nonzero scalar factor, $v'_{\l^{-i}}$ is equal to
$\prod_{\a\in\D_1^+}e_\a w_1$ with $w_1=f_\th^{\bar\l_0+1}v_\l$ as in the proof of Theorem \ref{verma-stru}.
The product of $e_\a$ can be ordered so that $\prod_{\a\in\D_1^+} e_\a=\prod_{\a\in\D_1^\pm}e_\a \prod_{\a\in\D_1^\ddag}e_\a$.  Then up to a nonzero scalar factor,
\begin{eqnarray}\label{highest--k-new-0}
v'_{\l^{-i}}&  =&
\prod_{\a\in\D_1^\pm}e_\a  \cdot  \prod_{\a\in\D_1^\pm}f_\a  \cdot \tilde{f} f_\th^{\bar\l_0+1-r_1}v_\l,
\end{eqnarray}
where $r_1=\#\D_1^\ddag$, and $\tilde f=f_\d$ if $\fg=G_3$ or $\tilde f=1$ else. This can be shown by noting that $[e_\a,f_\beta]=0$ for $\a\in\D_1^\ddag,\,\beta\in\D_1^\pm$,  and $[e_\a,f_\th]=f_{\bar \a}$ (up to a nonzero scalar factor) for $\a\in\D_1^\ddag$, where $\bar\a=w_0(\a)$ with $w_0\in W_0$ being the product of the all $\si_i$'s in \eqref{W0==}.  It is the unique root in $\D_1^\pm$ obtained from $\a$ by changing all signs of $\es_i$ with $i>0$ (cf.~\eqref{sys1}--\eqref{sys3}).

We use the same symbols $v'_{\l^{-i}}$ and $v_\l$ to denote the corresponding vectors in
the deformed parabolic Verma module $V_T(\l\wen{i})$.  Let
$D:=\langle v'_{\l^{-i}},v'_{\l^{-i}}\rangle$, which is a polynomial in $t$.
Since every nonzero submodule of $V(\l\wen{i})$ must contain the highest weight vector of the socle $L(\l^{-i})$  of  $V(\l\wen{i})$, we conclude that the lowest order term of the polynomial $D$ is of
a degree equal to the length $\ell$ of the Jantzen filtration,  that is, $D=t^{\ell} u(t)$ for some polynomial
$u(t)$ in $t$ with a nonzero constant term.

For convenience, we define
$\tilde e=e_\d$ if $\fg=G_3$ and else $\tilde e=1$.
Upon using \eqref{highest--k-new-0} (but interpreting $v'_{\l^{-i}}$ as in $V_T(\l\wen{i})$), we immediately see from \eqref{tau-prop} that
\begin{eqnarray}\label{highest--k-new-1+}
D v_\l & =  &e_\th^{\bar\l_0+1-r_1}\tilde e
\mbox{$\dis \prod_{\a\in\D_1^\pm}e_\a$}
\mbox{$\dis \prod_{\a\in\D_1^\pm}f_\a$}
v'_{\l^{-i}}.
\end{eqnarray}
Let $v''=\mbox{$ \prod_{\a\in\D_1^\pm}e_\a$}
\mbox{$ \prod_{\a\in\D_1^\pm}f_\a$}
v'_{\l^{-i}}$.
Observe from \eqref{highest--k-new-0} that $e_\a v'_{\l^{-i}}=0$ for all $\a\in\D_1^\pm$.  Similar computations as those in \eqref{WWWWW} show that up to a nonzero factor in $\C$,
\begin{eqnarray}\label{0hi--k-new-1+}
 v'' = p_1(t) v'_{\l^{-i}} \quad
\text{with \ \ } p_1(t):= \prod_{a\in\D_1^\pm} (\l^{-i}_{(t)}+\rho,\a),
\end{eqnarray}
where we have adopted the notation that
$\mu_{(t)}=\mu+t\d$  for any $\mu\in\fh^*$. Thus
\[
D v_\l =p_1(t)e_\th^{\bar\l_0+1-r_1}\tilde e v'_{\l^{-i}}.
\]

The proof of \eqref{0hi--k-new-1+} is done by case by case computations for
all the exceptional Lie superalgebras. Consider  as an example the case with $\fg=D(2,1;a)$.
Then up to nonzero  factors in $\C$,
\[
\begin{array}{rllll}
v''&=e_{\d-\es_1-\es_2}e_{\d-\es_1+\es_2}f_{\d-\es_1+\es_2}f_{\d-\es_1-\es_2}v'_{\l^{-i}}\\[4pt]
&=
e_{\d-\es_1-\es_2}[e_{\d-\es_1+\es_2},f_{\d-\es_1+\es_2}]f_{\d-\es_1-\es_2}v'_{\l^{-i}}\\[4pt]
&=\big(\d-\es_1+\es_2,\l^{-i}_{(t)}-(\d-\es_1-\es_2)\big)e_{\d-\es_1-\es_2}f_{\d-\es_1-\es_2}v'_{\l^{-i}}\\[4pt]
&= (\d - \es_1 + \es_2,\l^{-i}_{(t)} + \rho)[e_{\d-\es_1-\es_2},f_{\d-\es_1-\es_2}]v'_{\l^{-i}} \\[4pt]
&= (\d - \es_1 + \es_2,\l^{-i}_{(t)} + \rho)(\d - \es_1 - \es_2,\l^{-i}_{(t)})v'_{\l^{-i}}\\[4pt]
&= (\d - \es_1 + \es_2,\l^{-i}_{(t)} + \rho)(\d - \es_1 - \es_2,\l^{-i}_{(t)} + \rho)v'_{\l^{-i}}\\
&=\dis  \prod_{\a\in\D_1^\pm}(\a,\l^{-i}_{(t)}+\rho)v'_{\l^{-i}} = p_1(t) v'_{\l^{-i}},
\end{array}
\]
where the third and fifth equalities follow respectively from that
$(\d-\es_1+\es_2,-\d+\es_1+\es_2)=(\d-\es_1+\es_2,\rho)$ and $(\d-\es_1-\es_2,\rho)=0$ (cf.~\eqref{rho====}).

Now we return to an arbitrary exceptional Lie superalgebra.
Since $\l^{-i}=\l^{\si_0}$, we have $\l^{-i}_{(t)}+\rho=\si_0(\l^i+\rho)+t\d=\si_0(\l^i+\rho-t\d)$.  Hence
\begin{eqnarray}\label{P1ttt}
p_1(t) &\dis= \prod_{a\in\D_1^\pm}\big(\si_0(\l^{i}+\rho-t\d),\a\big)&=t p_{12}(t),
\end{eqnarray}
with $p_{12}(t)$ being a polynomial in $t$ with a nonzero constant term, where the last equality follows by noting 
the following facts. For $\a\in\D_1^\pm$,
$(\si_0(\l^{i}+\rho-t\d),\a\big)=(\l^{i}+\rho-t\d,\si_0(\a)\big)=(\l^{i}+\rho-t\d,-\bar\a\big)$, where $\bar\a$ is defined immediately after \eqref{highest--k-new-0}.  There exists a unique atypical root $\g$ of $\l$ such that $\g\in\D_1^\ddag$ since $\l$ is non-tail (cf.~\eqref{Tail-aty} and statements after it).

Next let us compute $e_\th^{\bar\l_0+1-r_1}\tilde e v'_{\l^{-i}}$.
Up to a nonzero factor in $\C$,
\begin{equation}\label{p2tt}
e_\th^{\bar\l_0+1-r_1}\tilde e v'_{\l^{-i}}= \tilde e
\prod_{\a\in\D_1^\pm}e_\a  e_\th^{\bar\l_0+1-r_1}f_\th^{\bar\l_0+1-r_1}w,
\end{equation}
where $w:=\prod_{\a\in\D_1^\pm}f_\a\tilde fv_\l$, with ``weight" $\mu_{(t)}:=\l^i_{(t)}-r_1\varsigma\d-2\d_{\fg,G_3}\d$ (where $\varsigma=\frac12$ if $\fg= F_4$ or $1$ else, $\d_{\fg,G_3}=1$ if $\fg=G_3$ or 0 else). Note that
$e_\th w=0$. Thus the factor $e_\th^{\bar\l_0+1-r_1}f_\th^{\bar\l_0+1-r_1}$ on the right-hand side of \eqref{p2tt} can be easily eliminated, leading to (up to a nonzero factor in $\C$)
 \begin{eqnarray}\label{highest--k-new-5}
t(\l_0+1-r_1)! \prod_{k=1}^{\l_0-r_1}(t+k) w',
\end{eqnarray}
where $w' =\mbox{$ \prod_{\a\in\D_1^\pm}e_\a$}\tilde e
\mbox{$\tilde f \prod_{\a\in\D_1^\pm}$} f_\a  v_\l$.
Calculations similar to those leading to \eqref{0hi--k-new-1+} and \eqref{WWWWW}
reduce $w'$ to
 \begin{eqnarray}\label{highest--k-new-6}
        w'=p_{22}(t)v_\l   \quad \text{with \ }p_{22}(t)=p_{21}(t) \prod_{\a\in\Psi}(\a,\l_t+\rho),
\end{eqnarray}
where $p_{21}(t)=1$ and $\Psi=\D_1^\pm$ if $\fg\ne G_3$, or else, $p_{21}(t)=\l_0+t-r_1$ (up to a nonzero factor in $\C$) and $\Psi=\D_1^\pm\bs\{\d\}$. For proving \eqref{highest--k-new-6}, one observes in the latter case that up to nonzero scalar factors, $\tilde e=e_\d$, $e_\d^2=e_{2\d}$, $f^2_{\d}=f_{2\d}$ and $[e_{2\d},f_{2\d}]=h_\th$; the factors $e_\d,\tilde e,\tilde f, f_\d$ in $w'$ are arranged to appear in this order. Note  that in the case $\fg=G_3$, $\atp^2=\emptyset$, and for $\l=\l^i\in\atp^1$ with $i\ge2$ we have $\l_0> 3=r_1$ by \eqref{l0-ige3}. Thus $p_{22}(t)$ is a polynomial with a nonzero constant term.

The long computation finally gives
\[
D=\langle v'_{\l^{-i}},v'_{\l^{-i}}\rangle = t^2 p_{12}(t) p_{22}(t) (\l_0+1-r_1)! \prod_{k=1}^{\l_0-r_1} (t+k),
\]
and $t^{-2} D$ is a polynomial with a nonzero constant term.
This shows that the Jantzen filtration in this case is of length $2$. By Lemma \ref{lem:graph}, this is the Loewy length. Hence the Jantzen filtration is a Loewy filtration, which is unique by Theorem \ref{thm:rigid}.

Next suppose $\l$ is in the first case of \eqref{Stru-Kac-mod}, i.e., $\l=\l^0\in\atp^2$.
Take $
v_+=\prod_{\a\in\D_1^\ddag}f_\a v_\l$  and  $v_-=\prod_{\a\in\D_1^\pm}f_\a v_\l.$
Then $\prod_{\a\in\D_1^\ddag}e_\a v_+=0=\prod_{\a\in\D_1^\pm}e_\a v_-$ as in the proof of \eqref{WWWWW}.
Arguments similar to those following \eqref{WWWWW} show that there exist some $u_\pm\in U(\fg_{-1})$ such that
$v'_{\l^{-1}_\pm}=u_\pm v_\l$ are primitive vectors with weights $\l^{-1}_\pm$ respectively.
Then the same methods used in the earlier part of this proof show that $v'_{\l^{-1}_\pm}\in V^1(\l)$, i.e.,  the Jantzen filtration has length 1, and hence is the unique Loewy filtration.

The other case can be proven in a similar way but much more simply as the primitive weight graph is simpler. We omit the details.

%
%
\subsection{Computation of $\fu^-$-homology groups}

In this section, we compute the homology groups $H_i(\fu^-, L(\lambda))$. The results will be needed for proving Theorem \ref{main-2}.  We remark that the results are interesting in their own right.
\begin{theorem}\label{main-3}
Let $\l$ be an atypical $\fg$-integral weight.
\begin{enumerate}
\item If $\l=\l\wen{i}\in\atp^1$ for some $i\in\Z^*$,
then as $\fg_0$-modules,
\begin{equation}\label{Main-ho}
H_k(\fu^-,L(\l\wen{i}))\cong\left\{\begin{array}{lll}
L^0(\l\wen{i-k})&\mbox{if \ }k=0\mbox{ or }i\le-1,     \\[4pt]
L^0(\l\wen{-k-1})&\mbox{if \ }k\ge i=1,\\[4pt]
L^0(\l\wen{-i-k})\oplus L^0(\l\wen{-k+i-2})&\mbox{if \ }k\ge i\ge2,\\[4pt]
L^0(\l\wen{-i-k})\oplus L^0(\l\wen{1})\oplus L^0(\l\wen{-1})&\mbox{if \ }k=i-1\ge1,\\[4pt]
L^0(\l\wen{-i-k})\oplus L^0(\l\wen{i-k})&\mbox{if \ }1\le k\le i-2.\end{array}\right.
\end{equation}
\item
If $\l=\l^i_+\in\atp^2$ for some $i\in\Z$,
then as $\fg_0$-modules,
\begin{equation}\label{Main-ho-2}
H_k(\fu^-,L(\l^{i}_+))\cong\left\{\begin{array}{lll}
L^0(\l^{i-k}_+)&\mbox{if \ }k=0\mbox{ or }i\le-1,\\[4pt]
L^0(\l^{-k}_+)\oplus L^0(\l^{-k}_-)&\mbox{if \ }k> i=0,\\[4pt]
L^0(\l^{i-k}_-)\oplus L^0(\l^{-i-k}_+)&\mbox{if \ }k\ge i\ge1,\\[4pt]
L^0(\l^{i-k}_+)\oplus L^0(\l^{-i-k}_+)&\mbox{if \ }1\le k<i
.\end{array}\right.
\end{equation}
\end{enumerate}
\end{theorem}
\begin{proof} We will prove (1) only as (2) can be proven similarly.
For any $\fu^-$-module $V$,  we denote
$H_k(V):=H_k(\fu^-,V)$ for simplicity.
Consider $V(\l^{-i})$ for $i\ge1$. Part (\ref{part2}) of Theorem \ref{verma-stru} gives the short exact sequence
\[
0\to L(\l\wen{-i-1})\to V(\l\wen{-i})\to L(\l\wen{-i})\to0,
\]
from which arises the following long exact sequence of homology groups:
\begin{eqnarray}\label{long-exact1}
                    \cdots \to H_k(L(\l\wen{-i-1}))\to&     H_k(V(\l\wen{-i}))    &\to H_k(L(\l\wen{-i}))\to\nonumber\\
                        \to  H_{k-1}(L(\l\wen{-i-1})\to &    H_{k-1}(V(\l\wen{-i}))    &\to H_{k-1}(L(\l\wen{-i}))\to\cdots.
\end{eqnarray}
Since the parabolic Verma module $V(\l)$ for any $\l$ is a free $\fu^-$-module, we always have (hereafter the ``equality'' always means the ``$\fg_0$-module isomorphism'') \begin{equation}\label{gen-Verma}H_k(V(\l))=\left\{
\begin{array}{ll}
L^0(\l)&\mbox{if }k=0,\\[4pt]
0&\mbox{otherwise},\end{array}
\right.\end{equation}
where $H_0(V(\l))$ is obtained from its definition. In fact,
$$H_0(V(\l))=V(\l)/\fu^-V(\l)=L^0(\l)=L(\l)/\fu^-L(\l)=H_0(L(\l)).$$
Thus \eqref{long-exact1} gives
\begin{eqnarray}\label{resu1}
H_k(L(\l\wen{-i}))=L^0(\l\wen{-i-k})\mbox{ \ for \ }k\ge0, i\ge1.
\end{eqnarray}
Similarly, from the short exact sequence $0\to L(\l\wen{-2})\to V(\l\wen{1})\to L(\l\wen{1})\to 0$ (cf.~the second graph of \eqref{Stru-Kac-mod}), we obtain
\begin{eqnarray}\label{resu2}
H_k(L(\l\wen{1}))=\left\{\begin{array}{llcc}
L^0(\l\wen{1})&\mbox{if \ }k=0,\\[4pt]
L^0(\l\wen{-k-1})&\mbox{otherwise.}
\end{array}\right.
\end{eqnarray}

Now consider $V(\l^i)$ with $i\ge2$.
We let $M_i,\,M_1$ be respectively the $\fg$-modules with primitive weight graphs
%
%
\begin{equation}\label{M-0000}
M_i:\ \ \l\wen{-i-1}\leftarrow \l\wen{i} \quad \text{and} \quad M_1:\ \ \l\wen{-1}\to\l\wen{-2} \leftarrow\l\wen{1}.
\end{equation}
 Then the third and fourth graphs of \eqref{Stru-Kac-mod} show that we have the short exact sequence \begin{equation}\label{Shor-222}0\to M_{i-1}\to V(\l\wen{i})\to M_i\to0.\end{equation}
Note that the subgraph of  $M_1$ obtained by deleting $\l\wen{-1}$ is the primitive weight graph for the parabolic Verma module $V(\l\wen{1})$ (cf.~the second graph of \eqref{Stru-Kac-mod}). Therefore, we have the short exact sequence
\[
0\to V(\l\wen{1})\to M_1\to L(\l\wen{-1})\to0,
\]
which gives rise to a long exact sequence of homology groups.
Since the homology groups of both $L(\l\wen{-1})$ and $V(\l\wen{1})$ are all known by \eqref{resu2} and \eqref{gen-Verma},  and in particular,$H_k(V(\l\wen{1}))=0$ for all $k>0$, this long exact sequence determines
\begin{eqnarray}\label{resu3}
H_k(M_1)=\left\{\begin{array}{llcc}
L^0(\l\wen{1})\oplus L^0(\l\wen{-1})&\mbox{if \ }k=0,\\[4pt]
L^0(\l\wen{-k-1})&\mbox{otherwise.}
\end{array}\right.
\end{eqnarray}
Analogously, from \eqref{Shor-222} and \eqref{gen-Verma}, we obtain for $i\ge2$,
\begin{eqnarray}\label{resu4}
H_k(M_i)=H_{k-1}(M_{i-1})=\left\{\begin{array}{llcc}
L^0(\l\wen{-k+i-2})&\mbox{if \ }k\ge i\ge2,\\[4pt]
L^0(\l\wen{1})\oplus L^0(\l\wen{-1})&\mbox{if \ }k=i-1\ge1,\\[4pt]
L^0(\l\wen{i-k})&\mbox{if \ }0\le k\le i-2.
\end{array}\right.
\end{eqnarray}
The primitive weight graph of $M_i$ for $i\ge2$ in \eqref{M-0000} yields the following short exact sequence: $0\to L(\l\wen{-i-1})\to M_i\to L(\l\wen{i})\to0$.
Thus we have the long exact sequence
\begin{eqnarray}\label{long-exact2}
                    \cdots \to H_k(L(\l\wen{-i-1}))\stackrel{\psi_k}{-\!\!\!-\!\!\!-\!\!\!\to}&     H_k(M_i)    &\to H_k(L(\l\wen{i}))\to\nonumber\\
                        \to  H_{k-1}(L(\l\wen{-i-1})\stackrel{\psi_{k-1}}{-\!\!\!-\!\!\!-\!\!\!\to} &    H_{k-1}(M_i)    &\to H_{k-1}(L(\l\wen{i}))\to\cdots.\end{eqnarray}
Note that all maps in \eqref{long-exact2} are $\fg_0$-module homomorphisms.
By inspecting \eqref{resu4}, we immediately see that as a $\fg_0$-module, $H_k(M_i)$ does not have a composition factor $L^0(\l\wen{-i-k-1})$.  Since
$H_k(L(\l\wen{-i-1}))=L^0(\l\wen{-i-k-1})$ by \eqref{resu1},
the $\fg_0$-module homomorphism $\psi_k$ must be zero. Hence
\begin{eqnarray}\label{resu5}
                H_k(L(\l\wen{i}))&   =   &H_{k-1}(L(\l\wen{-i-1}))\oplus H_{k}(M_{i})\nonumber\\
&   =   &\left\{\begin{array}{llcc}
L^0(\l\wen{-i-k})\oplus L^0(\l\wen{-k+i-2})&\mbox{if \ }k\ge i\ge2,\\[4pt]
L^0(\l\wen{-i-k})\oplus L^0(\l\wen{1})\oplus L^0(\l\wen{-1})&\mbox{if \ }k=i-1\ge1,\\[4pt]
L^0(\l\wen{-i-k})\oplus L^0(\l\wen{i-k})&\mbox{if \ }1\le k\le i-2\\[4pt]
L^0(\l\wen{i})&\mbox{if \ }0=k\le i-2.
\end{array}\right.
\end{eqnarray}
From this  together with \eqref{resu1} and \eqref{resu2}, we obtain \eqref{Main-ho}.
\end{proof}

\subsection{Proof of Theorem \ref{main-2}}\label{sect:proof-JKL}
%
%

Theorem \ref{main-2} is equivalent to
\begin{equation}\label{KL-1}\mbox{$\dis \sum_{\mu\in P_0^+}$}
J_{\l\mu}(q)p_{\mu\nu}(q)=\d_{\l\nu}\mbox{ \ for all }\l,\nu\in P_0^+.\end{equation}
Since the consecutive quotients $V(\l)_i$ of the Jantzen filtration are semisimple,
\[
\sum_\mu[V(\l)_i:L(\mu)][H_j(L(\mu)):L^0(\nu)]=[H_j(V(\l)_i):L^0(\nu)].
\]
Thus, the left-hand side of \eqref{KL-1} can be expressed as
$$\begin{array}{ll}
\mbox{$\dis \sum_{\mu,i,j}q^{i+j}(-1)^j[V(\l)_i:L(\mu)][H_j(L(\mu)):L^0(\nu)]$}\\[7pt]
\mbox{$\dis= \sum_{k}q^k\sum_{j=0}^k(-1)^j[H_j(V(\l)_{k-j}):L^0(\nu)]$}.\end{array}$$
Note that the constant term of the right-hand side of this equation is obviously equal to $\d_{\l\nu}$. Thus
the proof of \eqref{KL-1} is equivalent to showing
\begin{equation}\label{KL-2}
\sum_{j=0}^k(-1)^j H_j(V(\l)_{k-j})=0\mbox{ for }k\ge1,
\end{equation}
where the left-hand side is interpreted as an element in the Grothendieck group of the category $\cO_{\fg_0}$ of $U(\fg_0)$-modules.

We shall only consider in detail the case of an atypical $\l\in P^+$ such that either $\l=\l\wen{i}\in \atp^1$ for some $i\ge2$ or $\l=\l^i_\pm\in\atp^2$ with $i\ge1$, as  in all the other cases the parabolic Verma module $V(\lambda)$ has much simpler structure by Theorem \ref{verma-stru},  and the proof of \eqref{KL-2} is considerably easier.  Now assume $\l=\l^i\in\atp^1$ with $i\ge2$. Since the Jantzen filtration \eqref{filt} coincides with the radical filtration \eqref{FILT---}, we obtain from \eqref{Stru-Kac-mod} that
\[
V(\l\wen{i})_0 = L(\l\wen{i}),\ V(\l\wen{i})_1 = L(\l\wen{i-1}) \oplus  L(\l\wen{-i-1}) \oplus \d_{i2}L(\l\wen{-1}),\ V(\l\wen{i})_2 = L(\l\wen{-i}),
\]
and $V(\l\wen{i})_k=0$ for $k>2$.  Using the result \eqref{Main-ho} on $\fu^-$-homology groups,  one immediately obtains  \eqref{KL-2}.  The proof is similar if $\l=\l^i_\pm\in\atp^2$ with $i\ge1$.


\section{Characters, dimensions and cohomology groups of finite dimensional simple modules}

\subsection{Character and dimension formulae for simple modules}
Theorem \ref{verma-stru} enables us to derive a character formula and dimension formula for the
atypical finite-dimensional simple modules in a way analogous to the proofs of \cite[Theorem 4.4]{SZ4} and \cite[Theorem 4.16]{SZ2}. For an atypical weight $\l\in P^+$, we define
\equan{m-l-m} {S_\l = \{\l,\l^{{\si_0}}\} \cap \{\nu \in
P^+\mid  \nu \preccurlyeq \l\},\quad
m_\l = \#\big(\{\l,\l^{\si_0}\} \cap \{\nu \in
P^+\mid  \nu \succcurlyeq \l\}\big),} where $\l^{\si_0}$ is defined in \eqref{l-si-0}.
Then it is easy to see that
\[
S_\l=\left\{\begin{array}{l l}
\{\l,\l^{\si_0}\} & \  \text{if $\l\succ\l^{\si_0}\in P^+$}, \\
\{\l\}& \  \text{otherwise},
\end{array}\right.
\quad
m_{\l}= \left\{\begin{array}{l l}
2 &\ \text{if $\l\prec\l^{\si_0}$}, \\
1 & \  \text{otherwise}.
\end{array}\right.
\]
For $\mu\in S_\l$, we denote $\theta_{\l,\mu}\in W$ to be the
unique element with minimal length such that $\theta_{\l,\mu}\cdot\l=\mu$,
namely, $\theta_{\l,\mu}=1$ if $\l=\mu$ or $\theta_{\l,\mu}=\si_0$
otherwise.
Using the same method as that for the proof of \cite[Theorem 4.4]{SZ4}, we obtain
from Theorem \ref{verma-stru} the following result.

\begin{theorem}\label{main-theo1}Let $\fg$ be an exceptional Lie superalgebra, and
let $L(\l)$ be the finite-dimensional irreducible $\fg$-module with
atypical highest weight $\l$. 
\begin{enumerate}
\item The character of $L(\l)$ is given by
\equa{char-l-all}{\dis\ch L(\l)= \sum_{\mu\in
S_\l}{\dis\frac{(-1)^{|\theta_{\l,\mu}|}}{m_{\mu}R_{\bar0}}} \sum_{w\in
W}\sign{w}w\Big(e^{\mu+\rho_{\bar0}} \prod_{\b\in\D_1^+\bs\{\g_\mu\}}(1+e^{-\b})\Big),
}
where $R_{\bar0}$ is defined by \eqref{eq:RR},
$\g_\mu$ is the atypical root of $\mu\in S_\l$, and
$|\theta_{\l,\mu}|$ is the length of $\theta_{\l,\mu}$.
\item
The dimension of $L(\l)$ is given by
\equa{dim-l-all}{\dis{\rm dim\,} L(\l)= \sum_{\mu\in
S_\l,\,B\subset\D_1^+\bs\{\g_\mu\}}(-1)^{|\theta_{\l,\mu}|}m_{\mu}^{-1} \prod_{\a\in
\D_{\bar0}^+}\dis\frac{(\a,\rho_{\bar0}+\mu-\sum_{\b\in
B}\b)}{(\a,\rho_{\bar0})}.}
\item The character of the finite-dimensional $($typical
or atypical$)$ Kac $\fg$-module $K(\l)$ is $\ch K(\l)=\chi^V(\l)$ with
the right-hand side given by $(\ref{typical-char})$ unless
$\l^{\si_0}\in P^+$. If $\l^{\si_0}\in P^+$, then $\ch K(\l)=\ch L(\l)$ with
the right-hand side given by $(\ref{char-l-all})$.
\end{enumerate}
\end{theorem}

\subsection{First and second cohomology groups}
Applying the methods used in the proofs of \cite[Theorem 5.1]{SZ4} and \cite[Theorems 1.1--1.3]{SZ1} to the present case,
we obtain the following result from the structure theorem  (Theorem \ref{verma-stru}) of parabolic Verma modules.
\begin{theorem}\label{homology}
Let $\fg$ be an exceptional Lie superalgebra. Let $L(\l)$ and $K(\l)$ respectively denote
the finite-dimensional irreducible and Kac modules with
highest weight $\l$. Let $\L^{i}\in\atp^1,\,i=-1$ or $i\ge1$ be defined by $\L^{-1}=0,\,\L^1=(\L^{-1})\nex,\,\L^i=(\L^{i-1})\nex$ for $i\ge2$. Then
\begin{eqnarray*}
H^1(\fg,L(\l))&\cong&\left\{\begin{array}{lll}
	\C&\mbox{if \  }\l=\L^{2},\\
	0&\mbox{otherwise}.
\end{array}\right.\\
H^1(\fg,K(\l))&\cong&\left\{\begin{array}{lll}
	\C&\mbox{if \  }\l=\L^{3},\\
	0&\mbox{otherwise}.\end{array}\right.\\
H^2(\fg,L(\l))&\cong&\left\{\begin{array}{lll}
	\C&\mbox{if \  }\l=\L^{1},\,\L^{3},\\
	0&\mbox{otherwise}.\end{array}\right.\\
H^2(\fg,K(\l))&\cong&\left\{\begin{array}{lll}
	\C&\mbox{if \  }\l=\L^{1},\,\L^{4}.\\
	0&\mbox{otherwise}.\end{array}\right.
\end{eqnarray*}
\end{theorem}

\section{Comments on Jantzen filtration for orthosymplectic Lie superalgebras}
In this final section, we briefly comment on Jantzen filtration of parabolic Verma modules over
the remaining basic classical simple Lie superalgebras, the orthosymplectic Lie superalgebras ${\mathfrak{osp}}_{m|2n}$
with $m\ne 2$ (${\mathfrak{osp}}_{2|2n}$ is type I). Again we fix the distinguished maximal parabolic subalgebra $\fp$ of ${\mathfrak{osp}}_{m|2n}$ and consider the parabolic category $\cO^\fp$ of $\Z_2$-graded ${\mathfrak{osp}}_{m|2n}$-modules.

For $n=1$,  we can easily establish properties analogous to
Theorems \ref{main-1} and \ref{main-2}  by using results in \cite{SZ4}.  This was already alluded to in \cite{SZ3}.

\begin{theorem}\label{main-1+2}
The Jantzen filtration of the parabolic Verma module $V(\lambda)$ over ${\mathfrak{osp}}_{m|2}$
is the unique Loewy
filtration.
Furthermore, for any $\lambda, \mu\in P_0^+$,
the Jantzen polynomials $J_{\lambda \mu}(q)$ defined in \eqref{J-poly} coincide with the
inverse Kazhdan-Lusztig polynomials $a_{\lambda \mu}(q)$.
\end{theorem}
\begin{proof}
Theorem 4.2 in \cite{SZ4} is the precise analogue of Theorem \ref{verma-stru} for ${\mathfrak{osp}}_{k|2}$ with a slight change of notation.  Since the arguments in \S\ref{filtration-sec1} depend only on the primitive weight graphs in Theorem \ref{verma-stru}, they all go through in the present case, leading to the theorem. We omit the details.
\end{proof}

Finally for $\mathfrak{osp}_{m|2n}$ with $m\ne 2$ and $n>1$, super duality \cite{CLW} for
orthosymplectic Lie superalgebras will enable one to relate aspects of the Jantzen filtration for
parabolic Verma modules over $\mathfrak{osp}_{m|2n}$ to those of
the Jantzen filtration for parabolic Verma modules over orthogonal Lie algebras.
Even though the Jantzen filtration for the latter is not well understood except for
the cases corresponding to Hermitian symmetric pairs \cite{CIS, BB},
this nevertheless leads to useful insights into the problem at hand.
We will treat the Jantzen filtration
for the orthosymplectic Lie superalgebras in a future publication.

\medskip
\medskip

\noindent{\bf Acknowledgement}: This work was supported by the Australian Research
Council (grant no. DP0986551), the National Science Foundation of China (grant
no. 10825101), the Shanghai Municipal Science and Technology Commission (grant no.~12XD1405000) and
the Fundamental Research Funds for the Central Universities of China.


\begin{thebibliography}{9999}
\bibitem{A1} H.H. Andersen, {\em Filtrations of cohomology modules for Chevalley groups},
    Ann. Sci. \'Ecole Norm. Sup.  (4)  {\bf 16} (1983),  no. 4, 495--528

\bibitem{A} H.H. Andersen, {\em Jantzen's filtrations of Weyl modules},
     Math. Z.  {\bf 194}  (1987),
    no. 1, 127--142.

\bibitem{BB1}A.A. Beilinson and J. Bernstein, {\em
Localisation de $g$-modules},  C. R. Acad. Sci. Paris Ser. I Math.
{\bf 292}  (1981), no. 1, 15--18.

\bibitem{BB} A.A. Beilinson and J. Bernstein,
{\em A proof of Jantzen conjectures},
I. M. Gelfand Seminar,  1--50, Adv. Soviet Math., {\bf 16},
Part 1, Amer. Math. Soc.,
Providence, RI, 1993.
%

\bibitem{BC} B.D. Boe and D. H. Collingwood,
{\em Multiplicity free categories of highest weight representations. I, II},
Comm. Algebra  {\bf 18}  (1990),  no. 4, 947--1032, 1033--1070.

\bibitem{B} J. Brundan, {\em Kazhdan-Lusztig polynomials and character
formulae for the Lie superalgebra $\mathfrak{gl}(m|n)$}, J. Amer. Math. Soc {\bf
16} (2003), 185--231.

\bibitem{BS} J. Brundan and C. Stroppel,
{\em Highest weight categories arising from Khovanov's
diagram algebra IV: the general linear supergroup},
{J. Eur. Math. Soc. (JEMS) \bf 14} (2012), no. 2, 373--419.

\bibitem{BK}  J.-L. Brylinski and M. Kashiwara, {\em Kazhdan-Lusztig
conjecture and holonomic systems},
Invent. Math.  {\bf 64}  (1981), no. 3, 387--410.

\bibitem{CL} S.J. Cheng and N. Lam,
{\em Irreducible characters of general linear superalgebra and super
duality},
Comm. Math. Phys. {\bf298} (2010), 645--672.

\bibitem{CLW} S.J. Cheng, N. Lam and W. Wang,
{\em Super duality and irreducible
characters of ortho-symplectic Lie superalgebras},
Invent. Math. {\bf183} (2011), 189--224.

\bibitem{CWZ} S.J. Cheng, W. Wang and R.B. Zhang,
{\em Super duality and Kazhdan-Lusztig polynomials},
Trans. American Math. Soc. {\bf 360} (2008),
5883--5924.

\bibitem{CZ} S.J. Cheng and R.B. Zhang,
{\em Analogue of Kostant's u-cohomology formula for the general
linear superalgebra}, International Math. Research Notices (2004),
no. 1, 31--53.

\bibitem{CIS} D.H. Collingwood, R.S. Irving and B. Shelton,
{\em Filtrations on generalized Verma modules for Hermitian symmetric pairs},
J. Reine Angew. Math. {\bf 383} (1988), 54--86.

\bibitem{F} P. Fiebig, {\it Centers and translation functors for the category $\cO$ over Kac-Moody algebras},
 Math. Z. {\bf 243} (2003), no. 4, 689--717.

\bibitem{GJ} O. Gabber and A. Joseph,
{\em Towards the Kazhdan-Lusztig conjecture},
  Ann. Sci. \'Ecole Norm. Sup. (4)  {\bf 14}  (1981), no. 3, 261--302.
%

\bibitem{H} J.E. Humphreys, {\em Representations of semisimple Lie
algebras in the BGG category $\cO$}, Graduate Studies in Mathematics, {\bf 94},
American Mathematical Society, Providence, RI, 2008, xvi+289 pp.
%
%

\bibitem{I} R.S. Irving,  {\em A filtered category ${\cO}_S$ and applications}  (and {\em List of Errata}),
Mem. Amer. Math. Soc.  {\bf 83}  (1990),  no. 419, vi+117 pp.

\bibitem{J} J.C. Jantzen, {\em Kontravariante Formen auf
induzierten Darstellungen halbeinfacher Lie-Algebren},
    Math. Ann.  {\bf 226}  (1977), no. 1, 53--65.

\bibitem{J1} J.C. Jantzen, {\em Moduln mit einem h\"ochsten Gewicht},
Lecture Notes in Mathematics, {\bf 750},
    Springer, Berlin, 1979.

\bibitem{K} V.G.~Kac, {\em Lie superalgebras},
   Adv.~Math.~{\bf 26} (1977), 8--96.

\bibitem{Kac2} V.G.~Kac, {\em Characters of typical representations
of classical Lie superalgebras}, Comm. Alg. {\bf5} (1977), 889--897.

\bibitem{Kac3} V.G.~Kac, {\em Representations of classical Lie superalgebras},
Lect. Notes Math. {\bf676} (1978), 597--626.

\bibitem{KW} V.G.~Kac, M.~Wakimoto, {\em Integrable highest weight modules
over affine superalgebras and number theory}, in Lie Theory and
Geometry, 415--456, Progress in Math., 123, Birkhauser Boston,
Boston, MA, 1994.

\bibitem{KL} D. Kazhdan and G. Lusztig,
    {\em Representations of Coxeter groups and Hecke algebras},
    Invent. Math.  {\bf 53}  (1979), no. 2, 165--184.
%

\bibitem{Sch} M. Scheunert, {\em The theory of Lie superalgebras.
An introduction},
Lecture Notes in Mathematics, {\bf 716}, Springer, Berlin, 1979.

\bibitem{Se96} V.~Serganova, {\em Kazhdan-Lusztig polynomials and
       character formula for the Lie superalgebra $gl(m|n)$},
      Selecta Math.~{\bf 2} (1996), 607--654.


\bibitem{S} W. Soergel,  {\em Andersen filtration and hard Lefschetz},  Geom.
Funct. Anal.  {\bf 17}  (2008),  no. 6, 2066--2089.

\bibitem{Str} C. Stroppel, {\em Parabolic category $\mathcal O$, perverse sheaves on
Grassmannians, Springer fibres and Khovanov homology}, Compositio Math. {\bf 145} (2009), 954--992.

%

\bibitem{SHK} Y.~Su, J.W.B.~Hughes and
R.C.~King, {\em  Primitive vectors in the Kac-module of the Lie
superalgebra $sl(m|n)$}, J.~Math.~Phys.~{\bf41} (2000), 5044--5087.

\bibitem{SZ1} Y.~Su and R.B.~Zhang, {\em  Cohomology of Lie superalgebras
${\mathfrak{sl}}_{m|n}$ and ${\mathfrak{osp}}_{2|2n}$}, Proc.~London
Math.~Soc.~{\bf94} (2007), 91--136.

\bibitem{SZ2} Y.~Su and R.B.~Zhang, {\em Character and dimension formulae for
general linear superalgebra}, Adv.~Math.~{\bf211} (2007), 1--33.

\bibitem{SZ3}	Y.~Su and R.B. Zhang, {\em Generalised Jantzen filtration of Lie superalgebras I},
 J. Eur. Math. Soc. {\bf14} (2012), 1103--1133.

\bibitem{SZ4}	Y.~Su and R.B. Zhang, {\em Generalised Verma modules for the orthosymplectic Lie superalgebra ${\mathfrak{osp}}_{k|2}$}, Journal of Algebra {\bf357} (2012), 94--115.


%

\bibitem{V} J. Van der Jeugt,
{\em Irreducible representations of the exceptional Lie superalgebras $D(2,1;\alpha)$},
J. Math. Phys. {\bf  26} (1985), 913-924.

\bibitem{VZ} J. Van der Jeugt and R.B. Zhang,
{\em Characters and composition factor multiplicities
for the Lie superalgebra $gl(m/n)$}, Lett. Math. Physics, {\bf 47} (1999),
49--61.
%

\end{thebibliography}
\end{document}